\documentclass[10pt,twoside,reqno]{amsart}
\usepackage{amsmath,amsfonts,amsthm,amssymb}
\usepackage{graphicx}
\usepackage[all,cmtip,line]{xy}
\usepackage{array}
\newtheorem{theorem}{Theorem}[section]
\newtheorem{lemma}{Lemma}[section]
\newtheorem{proposition}{Proposition}[section]

\newtheorem{definition}{Definition}[section]
\newtheorem{remark}{Remark}[section]
\newtheorem{problem}{{\bf{Problem}}}

\theoremstyle{remark}

\newcommand \alp{\alpha}
\newcommand \eps{\varepsilon}
\newcommand \vphi{\varphi}

\newcommand \gam{\gamma}

\newcommand \R{\mathbb{R}}
\newcommand \der{\partial}

\newcommand \til{\tilde}

\newcommand \Sonic{\Gamma_{sonic}}
\newcommand \Shock{\Gamma_{shock}}
\newcommand \Wedge{\Gamma_{wedge}}
\newcommand \on{\text{on}}

\newcommand \inn{\text{in}}
\newcommand \For{\text{for}}
\newcommand \Om{\Omega}
\newcommand \vR{r}
\newcommand \mcl{\mathcal}

\newcommand{\dist}{ \mbox{dist}}
\newcommand{\divg}{ \mbox{div}}

\newcommand{\grad}{ D}

\newcommand{\PtUpL}{{P_1}}
\newcommand{\PtLwL}{{P_2}}
\newcommand{\PtLwR}{{P_3}}
\newcommand{\PtUpR}{{P_4}}

\newcommand{\NGl}{{\mathcal L}}
\newcommand{\Nl}{{\mathcal L}_1}
\newcommand{\Ml}{{\mathcal L}_2}

\newcommand{\Aa}{a}
\newcommand{\Ab}{b}
\newcommand{\Abeta}{\beta}
\newcommand{\AM}{M}

\newcommand{\Qr}{Q^+}

\numberwithin{equation}{section}

\begin{document}
\title[Regularity of Solutions to Regular Shock Reflection]
{Regularity of Solutions to Regular Shock Reflection for Potential
Flow}
\author{Myoungjean Bae}
\address{M. Bae, Department of Mathematics\\
         University of Wisconsin\\
         Madison, WI 53706-1388, USA}
\email{bae@math.wisc.edu}
\author{Gui-Qiang Chen}
\address{G.-Q. Chen,  Department of Mathematics\\
         Northwestern University \\
         Evanston, IL 60208-2730, USA}
\email{gqchen@math.northwestern.edu}
\author{Mikhail Feldman}
\address{M. Feldman, Department of Mathematics\\
         University of Wisconsin\\
         Madison, WI 53706-1388, USA}
\email{feldman@math.wisc.edu}
\subjclass{Primary: 35M10, 35J65, 35R35, 35J70, 76H05, 35B60, 35B65;
Secondary: 35L65, 35L67, 76L05}
\keywords{shock reflection, regular reflection
configuration, global solutions, regularity, optimal regularity,
existence, transonic flow, transonic shocks, free boundary problems,
degenerate elliptic, corner singularity, mathematical approach,
elliptic-hyperbolic, nonlinear equations, second-order, mixed type,
Euler equations, compressible flow}

\date{\today}
\begin{abstract}
The shock reflection problem is one of the most important problems
in mathematical fluid dynamics, since this problem not only arises
in many important physical situations but also is fundamental for
the mathematical theory of multidimensional conservation laws that
is still largely incomplete.
However, most of the
fundamental issues for shock reflection have not been understood,
including the regularity and transition of the different patterns of
shock reflection configurations. Therefore, it is important to
establish the regularity of solutions to shock reflection in order
to understand fully the phenomena of shock reflection.
On the other hand, for a regular reflection configuration, the
potential flow governs the exact behavior of the solution in
$C^{1,1}$ across the pseudo-sonic circle even starting from the full
Euler flow, that is, both of the nonlinear systems are actually the
same in an physically significant region near the pseudo-sonic
circle; thus, it becomes essential to understand the optimal
regularity of solutions for the potential flow across the
pseudo-sonic circle (the transonic boundary from the elliptic to
hyperbolic region) and at the point where the pseudo-sonic circle
(the degenerate elliptic curve) meets the reflected shock (a free
boundary connecting the elliptic to hyperbolic region). In this
paper, we study the regularity of solutions to regular shock
reflection for potential flow. In particular, we prove that the
$C^{1,1}$-regularity is {\it optimal} for the solution across the
pseudo-sonic circle and at the point where the pseudo-sonic circle
meets the reflected shock. We also obtain the $C^{2, \alpha}$
regularity of the solution up to the pseudo-sonic circle in the
pseudo-subsonic region. The problem involves two types of transonic
flow: one is a continuous transition through the pseudo-sonic circle
from the pseudo-supersonic region to the pseudo-subsonic region; the
other a jump transition through the transonic shock as a free
boundary from another pseudo-supersonic region to the
pseudo-subsonic region. The techniques and ideas developed in this
paper will be useful to other regularity problems for nonlinear
degenerate equations involving similar difficulties.
\end{abstract}
\maketitle

\section{Introduction}

We are concerned with the regularity of global solutions to shock
wave reflection by wedges. The shock reflection problem is one of
the most important problems in mathematical fluid dynamics,
which not only arises in many important physical situations but also
is fundamental for the mathematical theory of multidimensional
conservation laws that is still largely incomplete; its solutions
are building blocks and asymptotic attractors of general solutions
to the multidimensional Euler equations for compressible fluids (cf.
Courant-Friedrichs \cite{CF}, von Neumann \cite{Neumann},
Glimm-Majda \cite{GlimmMajda}, and Morawetz \cite{Morawetz2}; also
see \cite{BD,ChangChen,GlimmK,Guderley,LaxLiu,Serre,VD}).
%
%

In Chen-Feldman \cite{ChenFeldman}, the first global existence
theory of shock reflection configurations for potential flow has
been established when the wedge angle $\theta_w$ is large,
which converge to the unique solution of the normal reflection when
$\theta_w$ tends to $\pi/2$.
%
%
However, most of the fundamental issues for shock reflection by
wedges have not been understood, including the regularity and
transition of the different patterns of shock reflection
configurations. Therefore, it is important to establish the
regularity of solutions to shock reflection in order to understand
fully the phenomena of shock reflection, including the case of
potential flow which is widely used in aerodynamics (cf.
\cite{Bers1,CC,GlimmMajda,MajdaTh,Morawetz2}). On the other hand,
for the regular reflection configuration as in Fig. 1, the potential
flow governs the exact behavior of solutions in $C^{1,1}$ across the
pseudo-sonic ({\it sonic}, for short below) circle $\PtUpL\PtUpR$ even
starting from the full Euler flow, that is, both of the nonlinear
systems are actually the same in a physically significant region
near the sonic circle; thus, it becomes essential to understand the
optimal regularity of solutions for the potential flow across the
sonic circle $\PtUpL\PtUpR$
and at the point $\PtUpL$ where the sonic circle
meets the reflected shock.

\begin{figure}[h]
\centering
\includegraphics[height=2.4in,width=3.0in]{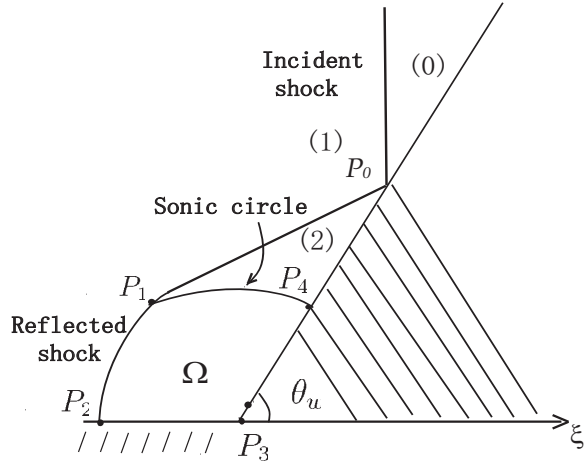}  
\caption[]{Regular Reflection Configuration}
\label{fig:RegularReflection}
\end{figure}

In this paper, we develop a mathematical approach in Sections 2--4
to establish the regularity of solutions to regular shock reflection
with the configuration as in Fig. 1 for potential flow. In
particular, we prove that the $C^{1,1}$-regularity is {\it optimal}
for the solution across the open part $\PtUpL\PtUpR$ of the sonic
circle (the degenerate elliptic curve) and at the point $\PtUpL$
where the sonic circle meets the reflected shock (as a free
boundary). The problem involves two types of transonic flow: one is
a continuous transition through the sonic circle $\PtUpL\PtUpR$ from
the pseudo-supersonic ({\it supersonic}, for short below) region (2)
to the pseudo-subsonic ({\it subsonic}, for short below) region
$\Omega$; the other is a jump transition through the transonic shock
as a free boundary from the supersonic region (1) to the subsonic
region $\Omega$. To achieve the optimal regularity, one of the main
difficulties is that the part $\PtUpL\PtUpR$ of the sonic circle is
the transonic boundary separating the elliptic region from the
hyperbolic region. Near  $\PtUpL\PtUpR$,  the solution is governed
by a nonlinear equation, whose main part has the form:
\begin{equation}\label{1.1}
(2x-a \psi_x)\psi_{xx} +b \psi_{yy} -\psi_x=0,
\end{equation}
where $a, b>0$ are constants, and which is elliptic in $\{x>0\}$
where $\psi>0$, with elliptic degeneracy at $\{x=0\}$ where $\psi =
0$. We analyze the features of equations modeled by \eqref{1.1} and
prove the $C^{2,\alpha}$ regularity of solutions of shock reflection
problem in the elliptic region up to the open part $\PtUpL\PtUpR$ of
the sonic circle. As a corollary, we establish that the
$C^{1,1}$-regularity is actually optimal across the transonic
boundary $\PtUpL\PtUpR$ from the elliptic to hyperbolic region.
Since the reflected shock $\PtUpL\PtLwL$ is regarded as a free
boundary connecting the hyperbolic region (1) with the elliptic
region $\Omega$ for the nonlinear second-order equation of mixed
type, another difficulty for the optimal regularity of the solution
is that the point $\PtUpL$ is exactly the one where the degenerate
elliptic curve $\PtUpL\PtUpR$ meets a transonic free boundary for
the nonlinear partial differential equation of second order. As far
as we know, this is the first optimal regularity result for
solutions to a free boundary problem of nonlinear degenerate
elliptic equations at the point where an elliptic degenerate curve
meets the free boundary. To achieve this, we construct two sequences
of points such that the corresponding sequences of values of
$\psi_{xx}$ have different limits at $\PtUpL$; this is done by
employing the $C^{2,\alpha}$ regularity of the solution up to
$\PtUpL\PtUpR$ excluding the point $\PtUpL$,
 and by studying detailed features of the free boundary
conditions on the free boundary $\PtUpL\PtLwL$, i.e., the
Rankine-Hugoniot conditions.

We note that the global theory of existence and
regularity of regular reflection configurations for the polytropic
case $\gamma>1$, established in \cite{ChenFeldman} and Sections
2--4, applies to the isothermal case $\gamma=1$ as well. The techniques and ideas
developed in this paper will be useful to other regularity problems
for nonlinear degenerate equations involving similar difficulties.

The regularity for certain degenerate elliptic and parabolic
equations has been studied (cf. \cite{Bet,DH,LW,Wu,Yang} and the
references cited therein). The main feature that distinguishes
equation (\ref{1.1}) from the equations in Daskalopoulos-Hamilton
\cite{DH} and Lin-Wang \cite{LW} is the crucial role of the
nonlinear term $-a\psi_x\psi_{xx}$. Indeed, if $\Aa=0$, then
(\ref{1.1}) becomes a linear equation
\begin{equation}
\label{mainEq-mainLinearTerms}
2x\psi_{xx}+\Ab\psi_{yy}-\psi_x=0.
\end{equation}
Then $\psi_0(x,y):=c x^{3/2}$ is a solution of
(\ref{mainEq-mainLinearTerms}), and $\psi_0$ with $c>0$ also
satisfies the conditions:
\begin{eqnarray}
&\psi>0 \qquad &\mbox{in}\,\,\{x>0\}, \label{1.3}\\
&\psi=0 \qquad &\mbox{on}\,\, \{x=0\}, \label{1.4}\\
&\partial_y\psi=0\qquad \,\,\, &\mbox{on}\,\, \{y=\pm 1\}.
\label{1.4-1}
\end{eqnarray}
Let $\psi$ be a solution of (\ref{mainEq-mainLinearTerms}) in
$Q_1:=(0,1)\times(-1,1)$ satisfying (\ref{1.3})--(\ref{1.4-1}). Then
the comparison principle implies that $\psi\ge\psi_0$ in $Q_1$ for
sufficiently small $c>0$. It follows that the solutions of
(\ref{mainEq-mainLinearTerms}) satisfying (\ref{1.3})--(\ref{1.4-1})
are not even $C^{1,1}$ up to $\{x=0\}$. On the other hand, for the
nonlinear equation (\ref{1.1}) with $\Aa>0$, the function
$\psi(x,y)=\frac{1}{2a}x^2$ is a smooth  solution of (\ref{1.1}) up
to $\{x=0\}$ satisfying (\ref{1.3})--(\ref{1.4-1}). More general
$C^{1,1}$ solutions of (\ref{1.1}) satisfying
(\ref{1.3})--(\ref{1.4-1}) and the condition
\begin{equation}\label{1.5}
-M x\le \psi_x\le \frac{2-\beta}{a}x \qquad \mbox{in}\,\,\, x>0,
\end{equation}
with $M>0$ and $\beta\in (0,1)$, which implies the ellipticity of
(\ref{1.1}), can  be constructed by using the methods of
\cite{ChenFeldman}, and their $C^{2,\alpha}$ regularity up to
$\{x=0\}$ follows from this paper. Another feature of the present
case is that, for solutions of (\ref{1.1}) satisfying
(\ref{1.3})--(\ref{1.4}) and (\ref{1.5}), we compute explicitly
$D^2\psi$ on $\{x=0\}$ to find that $\psi_{xx}(0,y)=\frac{1}{a}$ and
$\psi_{xy}(0,y)=\psi_{yy}(0,y)=0$ for all such solutions. Thus, all
the solutions are separated in $C^{2,\alpha}$ from the solution
$\psi\equiv 0$, although it is easy to construct a sequence of
solutions of (\ref{1.1}) satisfying  (\ref{1.3})--(\ref{1.4}) and
(\ref{1.5}) which converges to
 $\psi\equiv 0$ in $C^{1,\alpha}$.
This shows that the $C^{2,\alpha}$-regularity of (\ref{1.1}) with
conditions   (\ref{1.3})--(\ref{1.4}) and (\ref{1.5})
is a truly
nonlinear phenomenon of degenerate elliptic equations.

\smallskip Some efforts have been also made
mathematically for the shock reflection problem via simplified
models, including the unsteady transonic small-disturbance (UTSD)
equation (cf. Keller-Blank \cite{KB}, Hunter-Keller \cite{HK},
Hunter \cite{hunter1}, Morawetz \cite{Morawetz2}) and the pressure
gradient equation or the nonlinear wave system (cf. Zheng
\cite{Zheng1}, Canic-Keyfitz-Kim \cite{CKK1}). On the other hand, in
order to understand the existence and regularity of solutions near
the important physical points and for the reflection problem, some
asymptotic methods have been also developed (cf. Lighthill
\cite{Lighthill}, Keller-Blank \cite{KB}, Hunter-Keller \cite{HK},
Harabetian \cite{Harabetian}, and Morawetz \cite{Morawetz2}). Also
see Chen \cite{Sxchen} for a linear approximation of shock
reflection when the wedge angle is close to $\frac{\pi}{2}$
 and Serre
\cite{Serre} for an apriori analysis of solutions of shock
reflection and related discussions in the context of the Euler
equations for isentropic and adiabatic fluids. We remark that our
regularity results for potential flow near the sonic circle confirm
rigorously the asymptotic scalings used by Hunter-Keller \cite{HK},
Harabetian \cite{Harabetian}, and  Morawetz \cite{Morawetz2}.
Indeed, the $C^{2,\alpha}$ regularity up to the sonic circle away
from $\PtUpL$ and its proof based on the comparison with an ordinary
differential equation in the radial direction confirm their
asymptotic scaling via the differential equation in that region. The
optimal $C^{1,1}$ regularity at $\PtUpL$ shows that the asymptotic
scaling does not work there, i.e., the angular derivatives become
large, as stated in \cite{Morawetz2}.

\smallskip
The organization of this paper is the following. In Section 2, we
describe the shock reflection problem by a wedge and its solution
with regular reflection configuration when the wedge angle is
suitably large. In Section 3, we establish a regularity theory for
solutions near the degenerate boundary with Dirichlet data for a
class of nonlinear degenerate elliptic equations, in order to study
the regularity of solutions to the regular reflection problem. Then
we employ the regularity theory developed in Section 3 to establish
the optimal regularity of solutions for $\gamma>1$ across the sonic
circle $\PtUpL\PtUpR$ and at the point $\PtUpL$ where the sonic
circle $\PtUpL\PtUpR$ meets the reflected shock $\PtUpL\PtLwL$ in
Section 4. We also established the $C^{2,\alpha}$-regularity of
solutions in the subsonic region up to the sonic circle
$\PtUpL\PtUpR$. We further observe that the existence and regularity
results for regular reflection configurations for the polytropic
case $\gamma>1$ apply to the isothermal case $\gamma=1$.

\medskip
We remark in passing that there may exist a global regular
reflection configuration when state (2) is pseudo-subsonic, which is
in a very narrow regime (see \cite{CF,Neumann}). In this case, the
regularity of the solution behind the reflected shock is direct, and
the main difficulty of elliptic degeneracy does not occur.
Therefore, in this paper, we focus on the difficult case for the
regularity problem when state (2) is pseudo-supersonic, which will
be simply called a regular shock reflection configuration,
throughout this paper.

%

\medskip
\section{Shock Reflection Problem and Regular Reflection Configurations}
\label{RegNrSonicSect}

In this section, we describe the shock reflection problem by a wedge
and its solution with regular reflection configuration when the
wedge angle is suitably large.

The Euler equations for potential flow consist of the conservation
law of mass and the Bernoulli law for the density $\rho$ and the
velocity potential $\Phi$:
\begin{eqnarray}
&&\partial_t\rho + \divg_{\bf x}(\rho\nabla_{\bf x}\Phi)=0, \label{1.1.1} \\
&&\partial_t\Phi +\frac{1}{2}|\nabla_{\bf x}\Phi|^2+i(\rho)=B_0,
\label{1.1.2}
\end{eqnarray}
where $B_0$ is the Bernoulli constant determined by the incoming
flow and/or boundary conditions, and
$$
i'(\rho)=p'(\rho)/\rho=c^2(\rho)/\rho
$$
with $c(\rho)$ being the sound speed. For polytropic gas, by
scaling,
\begin{equation}\label{gamma-law}
p(\rho)=\rho^\gamma/\gamma,\qquad c^2(\rho)=\rho^{\gamma-1}, \qquad
i(\rho)=\frac{\rho^{\gamma-1}-1}{\gamma-1}, \qquad \gamma>1.
\end{equation}

\subsection{Shock Reflection Problem}

When a plane shock in the $({\bf x},t)$--coordinates, ${\bf
x}=(x_1,x_2)\in\R^2$, with left-state $ (\rho,\nabla_{\bf
x}\Phi)=(\rho_1, u_1,0) $ and right-state $(\rho_0, 0,0), u_1>0,
\rho_0<\rho_1$, hits a symmetric wedge
$$
W:=\{(x_1, x_2)\, :\, |x_2|< x_1 \tan\theta_w, x_1>0\}
$$
head on, it experiences a reflection-diffraction process. Then the
Bernoulli law \eqref{1.1.2} becomes
\begin{equation}\label{1.1.2b}
\partial_t\Phi +\frac{1}{2}|\nabla_{\bf x}\Phi|^2+i(\rho)=i(\rho_0).
\end{equation}
This reflection problem can be formulated as the following
mathematical problem.

\medskip
\begin{problem}[Initial-Boundary Value Problem]\label{IVP}  {\it Seek a
solution of the system of equations \eqref{1.1.1} and
\eqref{1.1.2b}, the initial condition at $t=0$:
\begin{equation}\label{initial-condition}
(\rho,\Phi)|_{t=0} =\begin{cases}
(\rho_0, 0) \qquad&  \mbox{for}\,\, |x_2|>x_1\tan\theta_w, x_1>0,\\
(\rho_1, u_1 x_1) \qquad &\mbox{for}\,\, x_1<0,
\end{cases}
\end{equation}
and the slip boundary condition along the wedge boundary $\partial
W$:
\begin{equation}\label{boundary-condition}
\nabla\Phi\cdot \nu|_{\partial W}=0,
\end{equation}
where $\nu$ is the exterior unit normal to $\partial W$ (see Fig.
{\rm 2}).}
\end{problem}

\begin{figure}[h]
\centering
\includegraphics[height=2.5in,width=2.5in]{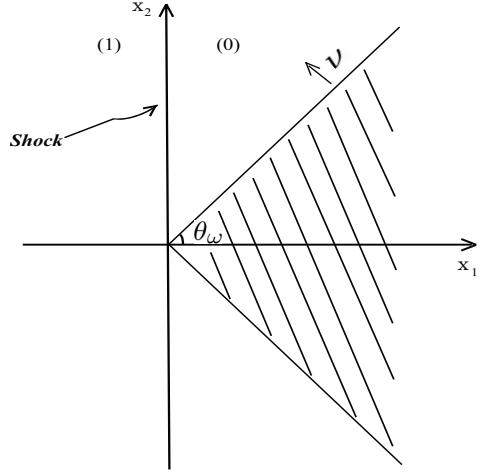}  
\caption[]{Initial-Boundary Value Problem} \label{fig:IBVP}
\end{figure}

\medskip
Notice that the initial-boundary value problem \eqref{1.1.1} and
\eqref{1.1.2b}--\eqref{boundary-condition} is invariant under the
self-similar scaling:
$$
({\bf x}, t)\to (\alpha {\bf x}, \alpha t), \quad (\rho, \Phi)\to
(\rho, \frac{\Phi}{\alpha}) \qquad \quad\mbox{for}\quad \alpha\ne 0.
$$
Thus, we seek self-similar solutions with the form
$$
\rho({\bf x},t)=\rho(\xi,\eta), \quad \Phi({\bf
x},t)=t\,\psi(\xi,\eta) \qquad\quad \mbox{for}\quad
(\xi,\eta)=\frac{{\bf x}}{t}.
$$
Then the pseudo-potential function
$\varphi=\psi-\frac{1}{2}(\xi^2+\eta^2)$
is governed by the following potential flow equation of second
order:
\begin{equation}
\divg\, \big(\rho(|D\varphi|^2, \varphi)D\varphi\big)
+2\rho(|D\varphi|^2, \varphi)=0 \label{1.1.5}
\end{equation}
with
\begin{equation}
\rho(|D\varphi|^2, \varphi)
=\big(\rho_0^{\gamma-1}-(\gamma-1)(\varphi+\frac{1}{2}|D\varphi|^2)\big)^{\frac
1{\gamma-1}}, \label{1.1.6}
\end{equation}
where the divergence $\divg$ and gradient $D$ are with respect to
the self-similar variables $(\xi,\eta)$. Then we have
\begin{equation}\label{c-through-density-function}
c^2=c^2(|D\varphi|^2,\varphi,\rho_0^{\gamma-1})
=\rho_0^{\gamma-1}-(\gamma-1)(\frac{1}{2}|D\varphi|^2+\varphi).
\end{equation}

Equation \eqref{1.1.5} is a nonlinear equation of mixed
elliptic-hyperbolic type. It is elliptic if and only if
\begin{equation}
|D\varphi| < c(|D\varphi|^2,\varphi,\rho_0^{\gamma-1}),
\label{1.1.8}
\end{equation}
which is equivalent to
\begin{equation}
|D \varphi| <c_*(\varphi, \rho_0, \gamma)
:=\sqrt{\frac{2}{\gamma+1}\big(\rho_0^{\gamma-1}-(\gamma-1)\varphi\big)}.
\label{1.1.8a}
\end{equation}
Shocks are discontinuities in the pseudo-velocity $D\varphi$.
 That
is, if $\Omega^+$ and $\Omega^-:=\Omega\setminus\overline{\Omega^+}$
are two nonempty open subsets of $\Omega\subset\R^2$ and
$S:=\partial\Omega^+\cap\Omega$ is a $C^1$--curve where $D\varphi$
has a jump,  then $\varphi\in W^{1,1}_{loc}(\Omega)\cap
C^1(\Omega^\pm\cup S)\cap C^2(\Omega^\pm)$ is a global weak solution
of (\ref{1.1.5}) in $\Omega$ if and only if $\varphi$ is in
$W^{1,\infty}_{loc}(\Omega)$ and satisfies equation \eqref{1.1.5} in
$\Omega^\pm$ and the Rankine-Hugoniot condition on $S$:
\begin{equation}\label{FBConditionSelfSim-0}
\left[\rho(|D\varphi|^2,\varphi)D\varphi\cdot\nu\right]_S=0.
\end{equation}

\medskip
The plane incident shock solution in the $({\bf x},t)$--coordinates
with states $(\rho, \nabla_{\bf x}\Psi)=(\rho_0, 0,0)$ and $(\rho_1,
u_1,0)$ corresponds to a continuous weak solution $\varphi$ of
(\ref{1.1.5}) in the self-similar coordinates $(\xi,\eta)$ with the
following form:
\begin{eqnarray}
&&\varphi_0(\xi,\eta)=-\frac{1}{2}(\xi^2+\eta^2) \qquad
 \hbox{for } \,\, \xi>\xi_0,
 \label{flatOrthSelfSimShock1} \\
&&\varphi_1(\xi,\eta)=-\frac{1}{2}(\xi^2+\eta^2)+ u_1(\xi-\xi_0)
\qquad
 \hbox{for } \,\, \xi<\xi_0,
 \label{flatOrthSelfSimShock2}
\end{eqnarray}
respectively, where
\begin{equation}\label{shocklocation}
\xi_0=\rho_1
\sqrt{\frac{2(\rho_1^{\gamma-1}-\rho_0^{\gamma-1})}{(\gamma-1)(\rho_1^2-\rho_0^2)}}
=\frac{\rho_1u_1}{\rho_1-\rho_0}>0
\end{equation}
is the location of the incident shock, uniquely determined by
$(\rho_0,\rho_1,\gamma)$ through (\ref{FBConditionSelfSim-0}), that
is, $P_0=(\xi_0, \xi_0 \tan\theta_w)$ in Fig. 1. Since the problem
is symmetric with respect to the axis $\eta=0$, it suffices to
consider the problem in the half-plane $\eta>0$ outside the
half-wedge
$$
\Lambda:=\{\xi<0,\eta>0\}\cup\{\eta>\xi \tan\theta_w,\, \xi>0\}.
$$
Then the initial-boundary value problem \eqref{1.1.1} and
\eqref{1.1.2b}--\eqref{boundary-condition} in the $({\bf x},
t)$--coordinates can be formulated as the following boundary value
problem in the self-similar coordinates $(\xi,\eta)$.

\medskip
\begin{problem}\label{BVP}{(Boundary Value Problem)} (see Fig. {\rm 1}). {\it Seek a
solution $\varphi$ of equation \eqref{1.1.5} in the self-similar
domain $\Lambda$ with the slip boundary condition on the wedge
boundary $\partial\Lambda$:
\begin{equation}\label{boundary-condition-3}
D\varphi\cdot\nu|_{\partial\Lambda}=0
\end{equation}
and the asymptotic boundary condition at infinity:
\begin{equation}\label{boundary-condition-2}
\varphi\to\bar{\varphi}:=
\begin{cases} \varphi_0 \qquad\mbox{for}\,\,\,
                         \xi>\xi_0, \eta>\xi \tan\theta_w,\\
              \varphi_1 \qquad \mbox{for}\,\,\,
                          \xi<\xi_0, \eta>0,
\end{cases}
\qquad \mbox{when $\xi^2+\eta^2\to \infty$},
\end{equation}}
where {\rm (\ref{boundary-condition-2})} holds in the sense that $
\displaystyle
\lim_{R\to\infty}\|\varphi-\overline{\varphi}\|_{C(\Lambda\setminus
B_R(0))}=0. $
\end{problem}
\subsection{Existence of Regular Reflection Configurations}

Since $\varphi_1$ does not satisfy the slip boundary condition
\eqref{boundary-condition-3}, the solution must differ from
$\varphi_1$ in $\{\xi<\xi_0\}\cap\Lambda$ and thus a shock
diffraction by the wedge occurs.  In \cite{ChenFeldman}, the
existence of global solution $\varphi$ to Problem 2 has been
established when the wedge angle $\theta_w$ is large,
and the corresponding structure of solution is as follows (see Fig.
\ref{fig:RegularReflection}): The vertical line is the incident
shock $S=\{\xi=\xi_0\}$ that hits the wedge at the point
$P_0=(\xi_0, \xi_0 \tan\theta_w)$, and state (0) and state (1) ahead
of and behind $S$ are given by $\varphi_0$ and $\varphi_1$ defined
in \eqref{flatOrthSelfSimShock1} and \eqref{flatOrthSelfSimShock2},
respectively. The solutions $\varphi$ and $\varphi_1$ differ within
$\{\xi<\xi_0\}$ only in the domain $P_0\PtUpL\PtLwL\PtLwR$ because
of shock diffraction by the wedge vertex, where the curve
$P_0\PtUpL\PtLwL$ is the reflected shock with the straight segment
$P_0\PtUpL$. State (2) behind $P_0\PtUpL$ is of the form:
\begin{equation}\label{state2a}
\varphi_2(\xi,\eta)=-\frac{1}{2}(\xi^2+\eta^2)+u_2(\xi-\xi_0)+
(\eta-\xi_0\tan\theta_w)u_2\tan\theta_w,
\end{equation}
which satisfies
$$
D\varphi\cdot\nu=0 \qquad \hbox{on }\, \partial\Lambda\cap
\{\xi>0\};
$$
the constant velocity $u_2$ and the angle between $P_0 \PtUpL$ and
the $\xi$--axis are determined by $(\theta_w,\rho_0,\rho_1,\gamma)$
from the two algebraic equations expressing
(\ref{FBConditionSelfSim-0}) and the continuous  matching of
$\varphi_1$ and $\varphi_2$ across $P_0 \PtUpL$. Moreover, the
constant density $\rho_2$ of state (2) satisfies $\rho_2>\rho_1$,
and state (2) is supersonic at the point $P_0$. The solution
$\varphi$ is subsonic within the sonic circle for state (2) with
center $(u_2, u_2\tan\theta_w)$ and radius $c_2=
\rho_2^{(\gamma-1)/2}>0$ (the sonic speed of state (2)), and
 $\varphi$ is
supersonic outside this circle containing the arc $\PtUpL\PtUpR$ in
Fig. \ref{fig:RegularReflection}, so that $\varphi_2$ is the unique
solution in the domain $P_0\PtUpL\PtUpR$, as argued in
\cite{ChangChen,Serre}. Then $\varphi$ differs from $\varphi_2$ in
the domain $\Omega=\PtUpL\PtLwL\PtLwR\PtUpR$, where the equation is
elliptic.

Introduce the polar coordinates $(r,\theta)$ with respect to the
center $(u_2, u_2\tan\theta_w)$ of the sonic circle of state (2),
that is,
\begin{equation}\label{coordPolar}
\xi-u_2=r\cos{\theta},\quad \eta-u_2\tan\theta_w=r\sin{\theta}.
\end{equation}
Then, for $\eps\in (0, c_2)$, we denote by
$\Omega_\eps:=\Omega\cap\{(r,\theta)\; : \;0<c_2-r<\eps\}$ the
$\eps$-neighborhood of the sonic circle $\PtUpL\PtUpR$ within
$\Omega$. In $\Omega_\eps$, we introduce the coordinates:
\begin{equation}\label{coordNearSonic}
x=c_2-r, \quad y=\theta-\theta_w.
\end{equation}
This implies that $\Omega_\eps\subset \{0<x<\eps, \; y>0\}$ and
$\PtUpL\PtUpR\subset \{x=0 \; y>0\}$. Also we introduce the
following notation for various parts of $\der\Om$:
\begin{align*}
&\Sonic:=\der \Om \cap \der B_{c_2}((u_2,u_2\tan\theta_w))\equiv
\PtUpL\PtUpR; \\
&\Shock:=\PtUpL\PtLwL;\\
&\Wedge:=\der \Om \cap \der \Lambda\equiv \PtLwR\PtUpR.
\end{align*}
Then the global theory established in \cite{ChenFeldman} indicates
that there exist $\theta_c=\theta_c(\rho_0,\rho_1,\gamma) \in (0,
\frac{\pi}{2})$ and $\alpha=\alpha(\rho_0,\rho_1,\gamma)\in (0,1)$
such that, when $\theta_w\in [\theta_c,\frac{\pi}{2})$, there exists
a global self-similar solution:
$$
\Phi({\bf x}, t) =t\,\varphi(\frac{\bf x}{t}) +\frac{|\bf x|^2}{2t}
\qquad\mbox{for}\,\, \frac{\bf x}{t}\in \Lambda,\, t>0,
$$
with
$$
\rho({\bf x}, t)=\big(\rho_0^{\gamma-1}-(\gamma-1)(\Phi_t
      +\frac{1}{2}|\nabla_{\bf x}\Phi|^2)\big)^{\frac{1}{\gamma-1}}
$$
of Problem {\rm 1} (equivalently, Problem {\rm 2}) for shock
reflection by the wedge, which satisfies that, for
$(\xi,\eta)=\frac{{\bf x}}{t}$,
\begin{equation}\label{phi-states-0-1-2}
\begin{split}
&\varphi\in C^{0,1}(\Lambda),\\
&\varphi\in C^{\infty}(\Omega)\cap C^{1,\alpha}(\bar{\Omega}),\\
&\varphi=\left\{\begin{array}{ll}
\varphi_0 \qquad\mbox{for}\,\, \xi>\xi_0 \mbox{ and } \eta>\xi\tan\theta_w,\\
\varphi_1 \qquad\mbox{for}\,\, \xi<\xi_0
  \mbox{ and above the reflection shock} \,\, P_0\PtUpL\PtLwL,\\
\varphi_2 \qquad \mbox{in}\,\, P_0\PtUpL\PtUpR.
\end{array}\right.
\end{split}
\end{equation}
Moreover,
\begin{enumerate}
\renewcommand{\theenumi}{\roman{enumi}}
\item \label{ellipticityInOmega}
equation \eqref{1.1.5} is elliptic in $\Omega$;

\item \label{phi-GE-phi2}
$
\vphi \ge \vphi_2\qquad\inn\;\;\Om$;

\item\label{BdryReg}
  the
reflected shock $P_0\PtUpL\PtLwL$ is $C^2$ at $\PtUpL$ and
$C^\infty$ except $\PtUpL$;

\item\label{C11NormEstimate}
there exists $\eps_0\in (0,\; \frac{c_2}{2})$ such that $\vphi\in
C^{1,1}(\overline{\Om_{\eps_0}})\cap C^2(\overline{\Om_{\eps_0}}
\setminus \overline\Sonic)$; moreover, in the coordinates {\rm
(\ref{coordNearSonic})},
\begin{equation}\label{parabolicNorm}
\|\vphi-\vphi_2\|^{(par)}_{2,0,\Om_{\eps_0}}
\!:=\!\sum_{0\le k+l
\le 2}\sup_{z\in\Om_{\eps_0}}(x^{k+\frac l2
-2}|\der_x^k\der_y^l(\vphi-\vphi_2)(x,y)|)<\infty;
\end{equation}
\item\label{psi-x-est}
there exists  $\delta_0>0$ so that, in the coordinates {\rm
(\ref{coordNearSonic})},
\begin{equation}\label{ell}
|\partial_x(\vphi-\vphi_2)(x,y)|\le
\frac{2-\delta_0}{\gam+1}x\qquad\;\inn \;\;\Om_{\eps_0};
\end{equation}
\item\label{FBnearSonic}
there exist $\omega>0$ and a function $y=\hat{f}(x)$ such that, in
the coordinates {\rm (\ref{coordNearSonic})},
\begin{equation}\label{OmegaXY}
\begin{split}
&\Om_{\eps_0}=\{(x,y)\,:\,x\in(0,\;\eps_0),\;\; 0< y<\hat{f}(x)\}, \\
&\Shock\cap\{0\le x\le \eps_0\}=\{(x,y):x\in(0,\;\eps_0),\;\;
y=\hat{f}(x)\},
\end{split}
\end{equation}
and
\begin{equation} \label{fhat}
\|\hat{f}\|_{C^{1,1}([0,\;\eps_0])}<\infty\;\;,\;\;
\frac{d\hat{f}}{dx}\!\ge\omega>0\!\;\;\; \text{for}\;\;0<x<\eps_0.
\end{equation}
\end{enumerate}

The existence of state (2) of the form (\ref{state2a}) with constant
velocity $(u_2, u_2\tan\theta_w)$, $u_2>0$, and constant density
$\rho_2>\rho_1$,  satisfying  (\ref{FBConditionSelfSim-0}) and
$\varphi_1=\varphi_2$ on $P_0 \PtUpL$, is shown in \cite[Section
3]{ChenFeldman} for $\theta_w\in [\theta_c,\frac{\pi}{2})$. The
existence of a solution $\varphi$ of Problem 2, satisfying
(\ref{phi-states-0-1-2}) and property (\ref{C11NormEstimate})
follows from \cite[Main Theorem]{ChenFeldman}. Property
(\ref{ellipticityInOmega}) follows from  Lemma 5.2 and  Proposition
7.1 in \cite{ChenFeldman}. Property (\ref{phi-GE-phi2}) follows from
Proposition 7.1 and Section 9 in \cite{ChenFeldman}, which assert
that $\varphi-\varphi_2\in\mathcal K$, where the set $\mathcal K$
defined by (5.15) in \cite{ChenFeldman}, which implies property
(\ref{phi-GE-phi2}). Property (\ref{psi-x-est}) follows from
Propositions 8.1--8.2 and Section 9 in \cite{ChenFeldman}. Property
(\ref{FBnearSonic}) follows from (5.7) and (5.25)--(5.27) in
\cite{ChenFeldman} and the fact that $\varphi-\varphi_2\in\mathcal
K$.

These results have been extended in \cite{ChenFeldman3} to other
wedge-angle cases.



\section{Regularity near the degenerate boundary for
nonlinear degenerate elliptic equations of second order}
\label{proofMainRegThmSect}

In order to study the regularity of solutions to the regular
reflection problem, in this section we first study the regularity of
solutions near a degenerate boundary for a class of nonlinear
degenerate elliptic equations of second order.

We adopt the following definitions for ellipticity and uniform
ellipticity: Let $\Omega\subset\R^2$ be open, $u\in C^2(\Omega)$,
and
\begin{equation}\label{2.1}
 \NGl u =\sum_{i,j=1}^2
A_{ij}(x, Du)u_{ij}+ B(x, Du),
\end{equation}
where $A_{ij}(x,p)$ and $B(x, p)$ are continuous on
$\overline{\Omega}\times\R^2$. The operator $\NGl$ is elliptic with
respect to $u$ in $\Omega$ if the coefficient matrix
$$
A(x,
Du(x)):=[A_{ij}(x, Du(x))]
$$
is positive for every $x\in\Omega$. Furthermore, $\NGl$ is uniformly
elliptic with respect to $u$ in $\Omega$ if
$$
\lambda I\le A(x, Du(x)) \le \Lambda I\qquad\mbox{for every
$x\in\Omega$},
$$
where $\Lambda\ge \lambda>0$ are constants and $I$ is the $2\times
2$ identity matrix.

The following standard comparison principle for the operator $\NGl$
follows from \cite[Theorem 10.1]{GilbargTrudinger}.

\begin{lemma} \label{comparison}
Let $\Omega\subset \R^2$ be an open bounded set. Let $u, v\in
C(\overline\Omega)\cap C^2(\Omega)$ such that the operator $\NGl$ is
elliptic in  $\Omega$ with respect to either $u$ or $v$. Let $\NGl u
\le \NGl v$ in $\Omega$ and $u\ge v$ on $\partial\Omega$.
 Then
$u\ge v$ in $\Omega$.
\end{lemma}

\subsection{Nonlinear Degenerate Elliptic Equations and Regularity Theorem}

We now study the regularity of positive solutions near the
degenerate boundary with Dirichlet data for the class of nonlinear
degenerate elliptic equations of the form:
\begin{align}
\label{mainEq-allTerms}
&\Nl\psi:=(2x-\Aa\psi_x+O_1)\psi_{xx}+O_2\psi_{xy}+(\Ab+O_3)\psi_{yy}-
(1+O_4)\psi_x+O_5\psi_y=0\quad \inn\;\Qr_{\vR, R},\\
\label{mainEq-positive}
&\psi>0\qquad\quad\inn\;\Qr_{\vR, R},\\
\label{mainEq-Dirichlet} &\psi=0\qquad\quad\on\;\partial\Qr_{\vR,
R}\cap\{x=0\},
\end{align}
where $\Aa, \Ab>0$ are constants and, for $\vR, R>0$,
\begin{equation}
\label{defRectangleNearSonic} \Qr_{\vR, R}:=\{(x, y)\;: \;x\in(0,
\vR),\;|y|<R\}\subset\R^2,
\end{equation}
and the terms $O_i(x,y), i=1, \dots, 5$ are continuously
differentiable and
\begin{align}
\label{Oks-int} &\frac{|O_1(x,y)|}{x^2},\;\frac{|O_k(x,y)|}{x}\le
N\qquad\quad&\text{for}\,\;k=2,\dots, 5, \\
\label{Oks-Der-int} &\frac{|DO_1(x,y)|}{x},\;|DO_k(x,y)|\le
N\qquad\quad&\text{for}\;\;k=2,\dots, 5,
\end{align}
in $\{x>0\}$  for some constant $N$.

Conditions (\ref{Oks-int})--(\ref{Oks-Der-int}) imply that the terms
$O_i, i=1, \cdots, 5$, are ``small''; the precise meaning of which
can be seen in Section \ref{OptReg-RegRefl-section} for the shock reflection problem below
(also see the estimates in \cite{ChenFeldman}). Thus, the main terms
of equation (\ref{mainEq-allTerms}) form the following equation:
\begin{equation}\label{mainEq-mainTerms}
(2x-\Aa\psi_x)\psi_{xx}+\Ab\psi_{yy}-\psi_x=0\qquad\quad\inn\;\Qr_{\vR,
R}.
\end{equation}

Equation (\ref{mainEq-mainTerms}) is elliptic  with respect to $\psi$
in $\{x>0\}$ if $\psi_x<\frac{2x}{\Aa}$. In this paper, we consider
the solutions that satisfy
\begin{equation}\label{x-deriv-small}
-Mx\le \psi_x\le{2-\Abeta\over \Aa}x\qquad\quad\inn\;\Qr_{\vR, R}
\end{equation}
for some constants $\AM\ge 0$ and $\Abeta\in(0,1)$. Then
(\ref{mainEq-mainTerms}) is uniformly elliptic in every subdomain
$\{x>\delta\}$ with $\delta>0$. The same is true for equation
(\ref{mainEq-allTerms}) in $\Qr_{\vR, R}$ if $\vR$ is sufficiently
small.

\begin{remark}\label{ellipticityRemark}
If $\hat{\vR}$
is sufficiently small, depending only on $\Aa, \Ab$, and $N$, then
{\rm (\ref{Oks-int})--(\ref{Oks-Der-int})} and {\rm
(\ref{x-deriv-small})} imply that equation {\rm
(\ref{mainEq-allTerms})} is uniformly elliptic with respect to
$\psi$ in $\Qr_{\hat{\vR}, R}\cap \{x>\delta\}$ for any
$\delta\in(0,\frac{\hat{\vR}}{2})$.
We will always assume such a choice of $\hat{\vR}$ hereafter.
\end{remark}

Let $\psi\in
C^2(\Qr_{\hat{\vR}, R})$ be a solution
of (\ref{mainEq-allTerms}) satisfying (\ref{x-deriv-small}).
Remark \ref{ellipticityRemark} implies that the interior regularity
\begin{equation}\label{interRegEq}
\psi\in C^{2,\alpha}(\Qr_{\hat{\vR}, R})\qquad\mbox{for all
}\;\alpha\in(0,1)
\end{equation}
follows first from the linear elliptic theory in two-dimensions (cf.
\cite[Chapter 12]{GilbargTrudinger}) to conclude the solution in
$C^{1,\alpha}$ which leads that the coefficient becomes $C^\alpha$
and then from the Schauder theory to get the $C^{2,\alpha}$ estimate
(cf. \cite[Chapter 6]{GilbargTrudinger}), where we use the fact
$O_i\in C^1(\{x>0\})$. Therefore, we focus on the regularity of
$\psi$ near the boundary $\{x=0\}\cap
\partial \Qr_{\hat{\vR},R}$ where the ellipticity of
(\ref{mainEq-allTerms}) degenerates.

\begin{theorem}[Regularity Theorem]\label{mainRegularityThm}
Let $\Aa, \Ab, M, N, R >0$ and $\Abeta\in (0, \frac{1}{4})$ be
constants. Let $\psi\in C(\overline{\Qr_{\hat{\vR}, R}})\cap
C^2(\Qr_{\hat{\vR}, R})$ satisfy {\rm (\ref{mainEq-positive})}--{\rm
(\ref{mainEq-Dirichlet})}, {\rm (\ref{x-deriv-small})}, and equation
{\rm (\ref{mainEq-allTerms})} in $\Qr_{\hat{\vR}, R}$ with
$O_i=O_i(x,y)$ satisfying $O_i\in C^1(\overline{\Qr_{\hat{\vR},
R}})$ and {\rm (\ref{Oks-int})}--{\rm (\ref{Oks-Der-int})}. Then
$$
\psi\in C^{2,\alpha}(\overline{\Qr_{\hat\vR/2, {R}/{2}}})
\qquad\mbox{for any}\,\, \alp\in (0,1),
$$
with
$$
\psi_{xx}(0,y)=\frac{1}{\Aa}, \quad \psi_{xy}(0,y)=\psi_{yy}(0,y)=0
\qquad \mbox{for all\,\, $|y|<\frac{R}{2}$}.
$$
\end{theorem}

To prove Theorem \ref{mainRegularityThm}, it suffices to show that,
for any given $\alpha\in (0,1)$,
\begin{equation}
\label{H} \psi\in
C^{2,\alpha}(\overline{\Qr_{r,R/2}})\qquad\;\text{for some}\; r\in
(0,\;\hat r/2),
\end{equation}
since $\psi$ belongs to $C^{2,\alpha}(\overline{\Qr_{ \hat
r/2,R/2}\cap \{x>r/2\}})$ by (\ref{interRegEq}).

%
%
%


Note that, by (\ref{mainEq-positive})--(\ref{mainEq-Dirichlet}) and
(\ref{x-deriv-small}), it follows that
\begin{equation}
\label{quadraticBound} 0<\psi(x,y)\le {2-\Abeta\over 2\Aa}x^2
\qquad\mbox{for all } \;(x,y)\in\Qr_{\hat{\vR}, R}.
\end{equation}

The essential part of the proof of Theorem \ref{mainRegularityThm}
is to show that, if a solution $\psi$ satisfies
(\ref{quadraticBound}), then, for any given $\alp\in (0,\;1)$, there
exists $r\in (0, \hat{r}/2]$ such that
\begin{equation}
\label{2-plus-alpha-Growth-Bound} |\psi(x,y)-{1\over 2\Aa}x^2|\le
Cx^{2+\alpha}  \qquad\mbox{for all } \;(x,y)\in\Qr_{r, 7R/8}.
\end{equation}
Notice that, although $\psi^{(0)}\equiv 0$ is a solution of
(\ref{mainEq-allTerms}),
it satisfies neither (\ref{2-plus-alpha-Growth-Bound}) nor the
conclusion $\psi^{(0)}_{xx}(0,y)=\frac{1}{\Aa}$ of Theorem
\ref{mainRegularityThm}. Thus it is necessary to improve first the
lower bound of $\psi$ in (\ref{quadraticBound}) to separate our
solution from the trivial solution $\psi^{(0)}\equiv 0$.

\subsection{Quadratic Lower Bound of $\psi$}
By Remark \ref{ellipticityRemark}, equation (\ref{mainEq-allTerms})
is uniformly elliptic  with respect to $\psi$ inside
$\Qr_{\hat{\vR}, R}$. Thus, our idea is to construct a positive
subsolution of (\ref{mainEq-allTerms}), which provides our desired
lower bound of $\psi$.

\begin{proposition} \label{h}
Let $\psi$ satisfy  the assumptions in Theorem {\rm
\ref{mainRegularityThm}}. Then there exist $r\in (0, \hat{r}/2]$ and
$\mu>0$, depending only on $\Aa, \Ab, N,  R, \hat{\vR}, \Abeta$, and
$\inf_{\Qr_{\hat{\vR}, R}\cap\{x>{\hat{\vR}}/{2}\}}\psi$, such that
$$
\psi(x,y)\ge \mu x^2\;\qquad\;\text{on}\;\;\Qr_{r, {15R}/{16}}.
$$
\end{proposition}

\begin{proof}
In this proof, all the constants below depend only on the data,
i.e., $\Aa, \Ab, M, N,  R, \hat{\vR}, \Abeta$, and
$\inf_{\Qr_{\hat{\vR}, R}\cap\{x>{\hat{\vR}}/{2}\}}\psi$, unless
otherwise is stated.

Fix $y_0$ with $|y_0|\le \frac{15R}{16}$. We now prove that
\begin{equation}\label{lowerQuadrBd}
\psi(x,y_0)\ge \mu x^2\;\;\qquad\text{for}\;\;x\in (0, r).
\end{equation}
We first note that,  without loss of generality, we may assume that
$R=2$ and $y_0=0$. Otherwise, we set
$\tilde\psi(x,y):=\psi(x,y_0+\frac{R}{32}y)$ for all $(x,y)\in
\Qr_{\hat{\vR}, 2}$. Then $\tilde\psi\in C(\overline{\Qr_{\hat{\vR},
2}})\cap C^2(\Qr_{\hat{\vR}, 2})$ satisfies equation
(\ref{mainEq-allTerms}) with (\ref{Oks-int}) and conditions
(\ref{mainEq-positive})--(\ref{mainEq-Dirichlet}) and
(\ref{x-deriv-small}) in $\Qr_{\hat{\vR}, 2}$, with some modified
constants  $\Aa, \Ab, N, \Abeta$ and functions $O_i$, depending only
on the corresponding quantities in the original equation and on $R$.
Moreover, $\inf_{\Qr_{\hat{\vR},
2}\cap\{x>{\hat{\vR}}/{2}\}}\tilde\psi=\inf_{\Qr_{\hat{\vR},
R}\cap\{x>{\hat{\vR}}/{2}\}}\psi$. Then (\ref{lowerQuadrBd}) for
$\psi$ follows from (\ref{lowerQuadrBd}) for $\tilde\psi$ with
$y_0=0$ and $R=2$. Thus we will keep the original notation with
$y_0=0$ and $R=2$. Then it suffices to prove
\begin{equation}\label{lowerQuadrBdzero}
\psi(x,0)\ge \mu x^2\;\;\qquad\text{for}\;\;x\in (0, r).
\end{equation}

By Remark \ref{ellipticityRemark} and the Harnack inequality, we
conclude that, for any $r\in (0, \hat{\vR}/2)$, there exists
$\sigma=\sigma(r)>0$ depending only on $r$ and the data $\Aa, \Ab,
N,  R, \hat{\vR}, \Abeta$, and $\inf_{\Qr_{\hat{\vR},
R}\cap\{x>{\hat{\vR}}/{2}\}}\psi$, such that
\begin{equation}\label{lowerConstBd}
\psi\ge\sigma\quad\qquad\text{on }\; \Qr_{\hat{\vR},
{3}/{2}}\cap\{x>r\}.
\end{equation}
Let $r\in (0, \hat{\vR}/2)$, $k>0$, and
\begin{equation}
\label{muDef} 0<\mu\le \frac{\sigma(r)}{r^2}
\end{equation}
to be chosen.
Set
\begin{equation}
\label{w-Def}
w(x,y):=\mu x^2(1-y^2)-kxy^2.
\end{equation}
Then, using (\ref{lowerConstBd})--(\ref{muDef}), we obtain that, for
all $x\in (0, r)$ and $|y|<1$,
\begin{equation*}\begin{cases}
w(0,y)=0 \le \psi(0,y),\\
w(r,y)\le \mu r^2 \le \psi(r,y),&   \\
w(x,\pm 1) =-kx \le 0 \le \psi(x,\pm 1).
\end{cases}\end{equation*}
Therefore, we have
\begin{equation}
\label{o} w\le \psi\;\qquad \text{on}\;\der\Qr_{r, 1}.
\end{equation}

Next, we show that $w$ is a strict subsolution  $\Nl w > 0$
 in  $\Qr_{r, 1}$,
 if the parameters are chosen appropriately.
In order to estimate $\Nl w$, we denote
\begin{equation}\label{defA}
 A_0:=\frac{k}{\mu}
 \end{equation}
and notice that
$$
w_{yy}=-2x(\mu x+k)=-2x(\mu x+k)\big((1-y^2)+y^2\big) =-2\mu
x(1-y^2)(x+A_0)-2ky^2x(\frac{x}{A_0}+1).
$$
Then, by a direct calculation and simplification, we obtain
\begin{equation} \label{a}
\Nl w=2\mu x(1\!-\!y^2)I_1+ky^2I_2,
\end{equation}
where
\begin{align*}
&I_1=
1\!-\!2\mu\Aa(1\!-\!y^2\!)\!-\!O_4\!+\!\frac{O_1}{x}\!-\!\big((\Ab\!+\!O_3)
-yO_5\big)(x\!+\!A_0)\!-\!\frac{y(2x\!+\!A_0)}{x}O_2,\\
&I_2= (1\!+\!O_4)\!+\!2\mu\Aa(1\!-\!y^2)\!-
\!2(\Ab\!+\!O_3)x(\frac{x}{A_0}\!+\!1)\!-\!2y
x(\frac{x}{A_0}\!+\!1)O_5\!- \!2y(\frac{2x}{A_0}\!+\!1)O_2.
\end{align*}
Now we choose $r$ and $\mu$ so that $\Nl w\ge 0$ holds. Clearly,
$\Nl w \ge 0$ if $I_1, I_2\ge 0$. By (\ref{Oks-int}), we find that,
in $\Qr_{r, 1}$,
\begin{equation}\label{I-termsEst}
\begin{split}
&I_1\ge 1-2\mu\Aa-C_0 r-(\Ab+N+C_0 r)A_0,\\
&I_2\ge 1-C_0 r-\frac{r}{A_0}C_0 r.
\end{split}
\end{equation}
Choose $r_0$ to satisfy the smallness assumptions stated above and
\begin{equation} \label{b}
0< r_0\le \min\{\frac{1}{4C_0},\frac{\Ab+N}{C_0},
\frac1{8\sqrt{C_0(\Ab +N)}}, \frac{\hat{\vR}}{2}\},
\end{equation}
where $C_0$ is the constant in (\ref{I-termsEst}). For such a fixed
$r_0$, we choose $\mu_0$ to satisfy (\ref{muDef}) and
\begin{equation}
\label{t1}
\mu_0 \le \frac{1}{8\Aa},
\end{equation}
and $A_0$ to satisfy
\begin{equation}
\label{t2} 4C_0 r_0^2<A_0<\frac{1}{8(\Ab+N)},
\end{equation}
where we have used (\ref{b}) to see that $4C_0
r_0^2<\frac{1}{8(\Ab+N)}$ in (\ref{t2}). Then $k$ is defined from
(\ref{defA}).
{}From (\ref{I-termsEst})--(\ref{t2}),
$$
I_1,\;I_2 > 0,
$$
which implies that
\begin{equation} \label{n}
\Nl w> 0\qquad \text{in }\;\Qr_{r, 1}
\end{equation}
whenever  $r\in (0,\; r_0]$ and $\mu\in(0,\;\mu_0]$.

By (\ref{o}), (\ref{n}),  Remark \ref{ellipticityRemark}, and the
comparison principle (Lemma \ref{comparison}), we have
$$
\psi(x,y)\ge w(x,y)=\mu x^2(1-y^2)-kxy^2\;\qquad\text{in}\;\;\Qr_{r,
1}.
$$
In particular,
\begin{equation}\psi(x,0)\ge \mu
x^2\;\qquad \text{for}\;\;x\in [0,\; r].
\end{equation}
This implies  (\ref{lowerQuadrBdzero}), thus (\ref{lowerQuadrBd}).
The proof is completed.
\end{proof}

 With Proposition \ref{h}, we now make the $C^{2,\alpha}$ estimate of
$\psi$.

\subsection{${\boldsymbol C^{2,\alpha}}$ Estimate of ${\boldsymbol \psi}$}
If $\psi$ satisfies
(\ref{mainEq-allTerms})--(\ref{mainEq-Dirichlet}) and
(\ref{x-deriv-small}), it is expected that $\psi$ is ``very close"
to $\frac{x^2}{2\Aa}$, which is a solution to
(\ref{mainEq-mainTerms}). More precisely,  we now prove
(\ref{2-plus-alpha-Growth-Bound}). To achieve this, we study the
function
\begin{equation}\label{DefWfunct}
W(x,y):=\frac{x^2}{2\Aa}-\psi(x,y).
\end{equation}
By (\ref{mainEq-allTerms}), $W$ satisfies
\begin{align}
\label{c}
\begin{split}
&\Ml W:=\!(x\!+\!\Aa W_x\!+\!O_1)W_{xx}\!\!+\!O_2W_{xy} \!\\
&\phantom{aaaaa}\,\,\,\,\, +\!(\Ab
\!+\!O_3)W_{yy}\!-\!(2\!+\!O_4)W_x\!+\!O_5W_y\!\!=\frac{O_1-xO_4}{\Aa}
\qquad \inn\;\Qr_{\hat{\vR}, R},
\end{split}
\\
&W(0,y)=0\qquad\qquad\qquad\qquad\qquad\qquad\qquad
\on\;\partial\Qr_{\hat{\vR}, R}\cap\{x=0\}, \label{Dirichlet-W}
\\
\label{deriv-W} &-{1-\Abeta\over \Aa}x\le W_x(x,y) \le (\AM+{1\over
\Aa})x \qquad\qquad \inn\;\Qr_{\hat{\vR}, R}.
\end{align}


\begin{lemma} \label{j}
Let $\Aa, \Ab, N, R, \hat{\vR}, \Abeta$, and $O_i$ be as in Theorem
{\rm \ref{mainRegularityThm}}. Let $\mu$ be the constant determined
in Proposition {\rm \ref{h}}. Then there exist $\alp_1\in (0,\;1)$
and $r_1>0$ such that, if $W\in C(\overline{\Qr_{\hat{\vR}, R}})\cap
C^2(\Qr_{\hat{\vR}, R})$ satisfies {\rm (\ref{c})}--{\rm
(\ref{deriv-W})}, then
\begin{equation}\label{estLemma-j}
W(x,y)\le \frac{1-\mu_1}{2\Aa r^{\alp}}x^{2+\alp}
 \;\;\qquad\text{in}\;\;\Qr_{r, 7R/8},
\end{equation}
whenever $\alp\in (0,\;\alp_1]$ and $r\in (0, r_1]$ with
$\mu_1:=\min(2a\mu, 1/2)$.
\end{lemma}

\begin{proof}
In the proof below, all the constants depend only on the data, i.e.,
$\Aa, \Ab, N, \Abeta, R, \hat{\vR}$, $\inf_{\Qr_{\hat{\vR},
R}\cap\{x>{\hat{\vR}}/{2}\}}\psi$, unless otherwise is stated.

By Proposition \ref{h},
\begin{equation} \label{upperbd}
W(x,y)\le\frac{1-\mu_1}{2\Aa}x^2 \qquad\quad\mbox{in $\Qr_{r_0,
15R/16}$},
\end{equation}
where $r_0$ depend only on $\Aa, \Ab, N, R, \hat{\vR}$, and
$\Abeta$.

Fix $y_0$ with $|y_0|\le \frac{7R}{8}$. We now prove that
$$
W(x,y_0)\le \frac{1-\mu_1}{2\Aa r^{\alp}}x^{2+\alp} \;\;\qquad
\text{for}\;\;x\in (0, r).
$$

By a scaling argument similar to the one in the beginning of proof
of Lemma \ref{h}, i.e., considering the function
$\tilde\psi(x,y)=\psi(x,y_0+\frac{R}{32}y)$ in $\Qr_{\hat{\vR}, 2}$,
 we conclude  that,  without loss of generality, we can assume that
$y_0=0$ and $R=2$.
That is, it suffices to prove that
\begin{equation}\label{W-Bd-atZero}
W(x,0)\le \frac{1-\mu_1}{2\Aa r^{\alp}}x^{2+\alp}
\;\;\qquad\text{for}\;\;x\in (0, r)
\end{equation}
for some $r\in (0, r_0), \alpha\in (0,\alpha_1)$, under the
assumptions that (\ref{c})--(\ref{deriv-W}) hold in $\Qr_{\hat{\vR},
2}$ and (\ref{upperbd}) holds in $\Qr_{r_0, 2}$.

For any given $r\in(0, r_0)$, let
\begin{align}
& A_1 r^{\alp}=\frac{1-\mu_1}{2\Aa},\quad B_1=\frac{1-\mu_1}{2\Aa},
\label{A-B-W-def}
\\
& v(x,y)=A_1x^{2+\alp}(1-y^2)+B_1x^2y^2. \label{V-def}
\end{align}
Since  (\ref{Dirichlet-W}) holds on $\partial\Qr_{\hat{\vR},
2}\cap\{x=0\}$ and (\ref{upperbd}) holds in  $\Qr_{r_0, 2}$, then,
for all $x\in(0, r)$ and $|y|\le 1$, we obtain
$$
\begin{cases}
v(0,y)=0 = W(0,y),\\
v(r,y)=\big(A_1 r^{\alp}(1-y^2)+B_1y^2\big) r^2
=\frac{1-\mu_1}{2\Aa} r^2 \ge W(r,y),\\
v(x,\pm 1)=B_1x^2=\frac{1-\mu_1}{2\Aa}x^2\ge W(x,\pm 1).
\end{cases}
$$
Thus,
\begin{equation}\label{bndryComparEl-2}
W\le v\;\qquad \on\;\der\Qr_{r, 1}.
\end{equation}

We now show that $\Ml v< \Ml W\;\;\text{in}\;\Qr_{r, 1}.$ From
(\ref{c}),
$$
\Ml v-\Ml W=\Ml  v-\frac{O_1-xO_4}{\Aa}.
$$
In order to rewrite the right-hand side in a convenient form, we
write the term $v_{yy}$ in the expression of $\Ml  v$ as
$(1-y^2)v_{yy}+ y^2 v_{yy}$ and use similar expressions for the
terms $v_{xy}$ and $v_y$. Then a direct calculation yields
$$
\Ml  v-\frac{O_1-xO_4}{\Aa}=
(2+\alp)A_1x^{1+\alp}(1-y^2)J_1+2B_1xy^2J_2,
$$
where
\begin{align*}
&J_1= (1+\alp)\Big(1+\Aa\big((2+\alp)A_1x^{\alp}(1-y^2)+
2B_1y^2\big)+\frac{O_1}{x}\Big)-(2+O_4)+T_1,\\
&J_2= 1+\Aa\big((2+\alp)A_1x^{\alp}(1-y^2)+2B_1y^2\big)+
\frac{O_1}{x}-(2+O_4)+T_2,\\
&T_1=\frac{1}{(2+\alp)A_1x^{1+\alp}}\Big(2O_2xy(2B_1-(2+\alp)A_1x^{\alp})+
2x^2(B_1-A_1x^{\alp})\big((\Ab+O_3)+O_5y\big)\\
&\qquad\qquad\qquad\qquad \qquad-
\frac{O_1-xO_4}{\Aa}\Big),\\
&T_2=\frac{(2+\alp)A_1x^{1+\alp}}{2B_1x}T_1.
\end{align*}
Thus, in $\Qr_{r, 1}$,
\begin{equation}
\label{f} \Ml v-\Ml W < 0\quad\qquad\text{if}\,\, J_1,\;J_2 < 0.
\end{equation}
By (\ref{Oks-int}) and (\ref{A-B-W-def}), we obtain
$$
|T_1|,\,|T_2|\le C r^{1-\alp} \qquad\,\,\mbox{in $\Qr_{r, 1}$},
$$
so that, in $\Qr_{r, 1}$,
\begin{align} \label{d}
&J_1\le (1+\alp)\big(1+\frac{2+\alp}{2}(1-\mu_1)\big)-2 +C r^{1-\alp},\\
\label{e} &J_2\le 1+\frac{2+\alp}{2}(1-\mu_1)-2+C r^{1-\alp}.
\end{align}
Choose $\alp_1>0$, depending only on $\mu_1$, so that, if
$0<\alp\le\alp_1$,
\begin{align}
\label{choseAlpha0}
(1+\alp)\Big(1+\frac{2+\alp}{2}(1-\mu_1)\Big)-2
\le -\frac{\mu_1}{4}.
\end{align}
Such a choice of $\alp_1>0$ is possible because we have the strict
inequality in (\ref{choseAlpha0}) when $\alp=0$, and the left-hand
side is an increasing function of $\alp>0$ (where we have used
$0<\mu_1\le 1/2$ by reducing $\mu$ if necessary). Now, choosing
$r_1>0$ so that
\begin{align}
\label{chooseEps0}
r_1<\min\big\{\left(\frac{\mu_1}{4C}\right)^\frac{1}{1-\alp},\;
r_0\big\}
\end{align}
is satisfied,
we use (\ref{d})--(\ref{choseAlpha0}) to obtain
$$
J_1, J_2 <0 \quad\qquad\text{in }\; \Qr_{r, 1}.
$$
Then, by
$(\ref{f})$, we obtain
\begin{equation} \label{p}
\Ml v< \Ml W\;\;\qquad\text{in}\;\;\Qr_{r, 1}
\end{equation}whenever
$r\in (0,\;r_1]$ and $\alp\in(0,\;\alp_1]$. By
(\ref{bndryComparEl-2}), (\ref{p}), Remark \ref{ellipticityRemark},
and the standard comparison principle (Lemma \ref{comparison}), we
obtain
\begin{equation}
W \le v \quad\qquad \text{in } \;\Qr_{r, 1}.
\end{equation}
In particular, using (\ref{A-B-W-def})--(\ref{V-def}) with $y=0$, we
arrive at (\ref{W-Bd-atZero}).
\end{proof}

Using Lemma \ref{j}, we now generalize the result (\ref{estLemma-j})
for any $\alp\in (0, 1)$.

\begin{proposition}\label{proposition1}
Let $\Aa, \Ab, N, R, \hat{\vR}, \Abeta$, and $O_i$ be as in Theorem
{\rm \ref{mainRegularityThm}}. Then, for any $\alp\in (0, 1)$, there
exist positive constants $r$ and $A$ which depend only on
$a,b,N,R,\hat{r},\beta$, and $\alp$ so that, if $W\in
C(\overline{Q^+_{\hat{r},R}})\cap C^2(Q^+_{\hat{r},R})$ satisfies
{\rm (\ref{c})}--{\rm (\ref{deriv-W})}, then
\begin{equation}\label{2-all alpha-growth}
W(x,y)\le A x^{2+\alp}\qquad\;\text{in}\;\;Q^+_{r, 3R/{4}}.
\end{equation}
\end{proposition}

\begin{proof}
As argued before, without loss of generality, we may assume that
$R=2$ and it suffices to show that
\begin{equation}\label{main-ineq-any alpha}
W(x,0)\le Ax^{2+\alp}\;\;\qquad \text{for}\;\;x\in [0,\; r].
\end{equation}
 By
Lemma \ref{j}, it suffices to prove \eqref{main-ineq-any alpha} for
the case $\alp>\alp_1$. Fix any $\alp\in (\alp_1,\;1)$ and set the
following comparison function:
\begin{equation}\label{comparison-ftn-any alpha}
u(x,y)=\frac{1-\mu_1}{2a r_1^{\alp_1}
r^{\alp-\alp_1}}x^{2+\alp}(1-y^2) +\frac{1-\mu_1}{2a
r_1^{\alp_1}}x^{2+\alp_1}y^2.
\end{equation}
By Lemma \ref{j},
\begin{equation}\label{comparison-any alp-bdry}
W\le u\;\qquad\text{on}\;\der Q^+_{r,1} \quad\mbox{for}\, \, r\in
(0, r_1].
\end{equation}
As in the proof of Lemma \ref{j}, we write
$$
\mcl{L}_2u-\frac{O_1-xO_4}{a} u
=(2+\alp)\frac{(1-\mu_1)x^{1+\alp}}{2a r_1^{\alp_1}
r^{\alp-\alp_1}}(1-y^2)\hat{J}_1
+(2+\alp_1)\frac{(1-\mu_1)x^{1+\alp_1}}{2a
r_1^{\alp_1}}y^2\hat{J}_2,
$$
where
\begin{align*}
&D_0=\frac{1-\mu_1}{2}\Big((1-y^2)(2+\alp)\bigl(\frac{x}{r}\bigr)^{\alp}
+y^2(2+\alp_1)\bigl(\frac{x}{r}\bigr)^{\alp_1}\Big),\\
&\hat{J}_1=(1+\alp)\Big(1+
\bigl(\frac{r}{r_1}\bigr)^{\alp_1}D_0\Big)-2+\hat{T}_1,\\
&\hat{J}_2=(1+\alp_1)\Big(1+\bigl(\frac{r}{r_1}\bigr)^{\alp_1}D_0\Big)-2+\hat{T}_2,\\
&\hat{T}_1=\frac{2a r_1^{\alp_1}
r^{\alp-\alp_1}}{(2+\alp(1-\mu_1))x^{1+\alp}}
\Big(\mcl{L}_2u-\big((x+au_x)u_{xx}-2u_x\big)-\frac{O_1-xO_4}{a}\Big),\\
&\hat{T}_2=\frac{2a r_1^{\alp_1}}{(a+\alp_1)(1-\mu_1)x^{1+\alp_1}}
\Big(\mcl{L}_2u-\big((x+au_x)u_{xx}-2u_x\big)-\frac{O_1-xO_4}{a}\Big).
\end{align*}
By (\ref{Oks-int}), we have
$$
|\hat{T}_1|,|\hat{T}_2|\le C r^{1-\alp_1}
$$
for some positive constant $C$ depending only on $a,b,N,\beta,r_1$,
and $\alp_1$. Thus, we have
\begin{equation}
\label{bar J estimate} \max\{\hat{J}_1,\hat{J}_2\}
\le(1+\alp)\Big(1+\bigl(\frac{r}{r_1}\bigr)^{\alp_1}(2+\alp)
\frac{1-\mu_1}{2}\Big)-2+C r^{1-\alp_1} \qquad\mbox{for $(x,y)\in
Q^+_{r,1}$}.
\end{equation}

 Choosing $r>0$ sufficiently
small, depending only on $r_1, \alp$, and $C$, we obtain
$$
\mcl{L}_2u-\mcl{L}_2W=\mcl{L}_2u-\frac{O_1-xO_4}{a}<
0\;\qquad\text{in}\;\;Q^+_{r,1}.
$$
Then Lemma \ref{comparison} implies that
$$
W\le u\;\qquad\text{in}\;\;Q^+_{r,1}.
$$
Thus, (\ref{main-ineq-any alpha}) holds with
$$
A=\frac{1-\mu_1}{2a r_1^{\alp_1} r^{\alp-\alp_1}}.
$$
\end{proof}

\begin{lemma} \label{k}
Let $\Aa, \Ab, N, R, \hat{\vR}, \Abeta$, and $O_i$ be as in Theorem
{\rm \ref{mainRegularityThm}}. Then there exist $r_2>0$ and
$\alp_2\in(0, 1)$ such that, if $W\in C(\overline{\Qr_{\hat{\vR},
R}})\cap C^2(\Qr_{\hat{\vR}, R})$ satisfies {\rm (\ref{c})}--{\rm
(\ref{deriv-W})}, we have
\begin{equation}\label{estLemma-k}
W(x,y)\ge -\frac{1-\Abeta}{2\Aa{r}^{\alp}}x^{2+\alp}
\;\;\qquad\text{in}\;\;\Qr_{r, {7R}/{8}}
\end{equation}
whenever $\alp\in (0,\;\alp_2]$ and $r\in(0,\;r_2]$.
\end{lemma}

\begin{proof} By (\ref{Dirichlet-W})--(\ref{deriv-W}), it can  be easily
verified that $W(x,y)\ge-\frac{1-\Abeta}{2\Aa}x^2$ in
$\Qr_{\hat{\vR}, R}$. Now, similar to the proof of  Lemma \ref{j},
it suffices to prove that, with the assumption $R=2$,
$$
W(x,0)\ge -\frac{1-\Abeta}{2\Aa r^{\alp}}x^{2+\alp}
\;\qquad\text{for}\;\;x\in (0, r)
$$
for some $r>0$ and $\alp\in (0,\alpha_2)$.

For this, we use the comparison function:
$$
v(x,y):=-Lx^{2+\alp}(1-y^2)-Kx^2 y^2,\;\quad\text{with}\;\;
L{r}^{\alp}=K=\frac{1-\Abeta}{2\Aa}.
$$
%
Then we follow the same procedure as the proof of Lemma \ref{j},
except that $\Ml v> \Ml W$, to find that the conditions for the
choice of $\alp, r>0$ are inequalities
(\ref{choseAlpha0})--(\ref{chooseEps0}) with $\mu_1, r_1$ replaced
by $\Abeta, r_2$ and with an appropriate constant $C$.
\end{proof}

Using Lemma \ref{k}, we now generalize the result \eqref{estLemma-k} for
any $\alpha\in (0,1)$.

\begin{proposition}\label{proposition2}
Let $\Aa, \Ab, N, R, \hat{\vR}, \Abeta$, and $O_i$ be as in Theorem
{\rm \ref{mainRegularityThm}}. Then, for any $\alp\in (0,\;1)$,
there exist positive constants $r$ and $B$ depending on $a,b,N,R,
\hat{r},\beta,$ and $\alp$ so that, if $W\in
C(\overline{Q^+_{\hat{r},R}})\cap C^2(Q^+_{\hat{r},R})$ satisfies
{\rm (\ref{c})}--{\rm (\ref{deriv-W})}, then
\begin{equation}\label{3-all alpha-growth}
W(x,y)\ge -B x^{2+\alp}\;\qquad \text{in}\;\;Q^+_{r,3R/4}.
\end{equation}
\end{proposition}

\begin{proof}
For fixed $\alp\in (\alp_2, 1)$, we set the comparison function:
$$
u_-(x,y)=-\frac{1-\beta}{2a r_2^{\alp_2}
r^{\alp-\alp_2}}x^{2+\alp}(1-y^2) -\frac{1-\beta}{2a
r_2^{\alp_2}}x^{2+\alp_2}y^2.
$$
Then, using the argument as in the proof of Proposition
\ref{proposition1}, we can choose $r>0$ appropriately small so that
$$
\mcl{L}_2u_--\mcl{L}_2W=\mcl{L}_2 u_--\frac{O_1-xO_4}{a}> 0
$$
holds for all $(x,y)\in Q^+_{r,1}$.
\end{proof}

With Proposition \ref{proposition1}--\ref{proposition2}, we now
prove Theorem 3.1.

\subsection{Proof of Theorem \ref{mainRegularityThm}}
We divide the proof into four steps.

Step 1. Let $\psi$ be a solution of (\ref{mainEq-allTerms}) in
$\Qr_{\hat{\vR}, R}$ for $\hat{r}$ as in Remark
\ref{ellipticityRemark}, and let the assumptions of Theorem
\ref{mainRegularityThm} hold.
Then $\psi$ satisfies (\ref{interRegEq}). Thus it suffices to show
that, for any given $\alp\in (0,1)$, there exists $r>0$ so that
$\psi\in C^{2,\alpha}(\overline{\Qr_{r, {R}/{2}}})$ and
$\psi_{xx}(0, y)=\frac{1}{\Aa}$, $\psi_{xy}(0,y)=\psi_{yy}(0,y)=0$
for all $|y|<\frac{R}{2}$.

Let $W(x,y)$ be defined by (\ref{DefWfunct}). Then, in order to
prove Theorem \ref{mainRegularityThm}, it suffices to show that, for
any given $\alp\in (0,\;1)$, there exists $r>0$ so that

(i) $W\in C^{2,\alpha}(\overline{\Qr_{r, {R}/{2}}})$;

(ii) $D^2W(0, y)=0$ for all $|y|<\frac{R}{2}$.

\medskip
Step 2. By definition, $W$ satisfies (\ref{c})--(\ref{deriv-W}). For
any given $\alp\in (0,1)$, there exists $r>0$ so that both
(\ref{2-all alpha-growth}) and (\ref{3-all alpha-growth}) hold in
$Q^+_{r,3R/4}$ by Propositions
\ref{proposition1}--\ref{proposition2}. Fix such $r>0$.

Furthermore, since $W$ satisfies estimate (\ref{deriv-W}), we can
introduce a cutoff function into the nonlinear term of equation
(\ref{c}), i.e., modify the nonlinear term away from the values
determined by (\ref{deriv-W}) to make the term bounded in $W_x/x$.
Namely, fix $\zeta\in C^\infty(\R)$ satisfying
\begin{equation}\label{defZeta}
\begin{split}
&-{1-\Abeta/2\over \Aa}\le\zeta\le M+{2\over \Aa}\;\;\,\, \on\;
\R;\qquad\,\,
\zeta(s)=s\;\;\,\,\on\; \big(-{1-\Abeta\over \Aa},\;M+{1\over \Aa}\big);\\
& \zeta\equiv 0\;\qquad \on\;\;\R\setminus\big(-{{2-\Abeta}\over
\Aa},\;M+{4\over \Aa}\big).
\end{split}
\end{equation}
Then, from (\ref{c}) and (\ref{deriv-W}), it follows that $W$
satisfies
\begin{equation}
\label{c-cutoff}
\begin{split}
&x\big(1\!+\!\Aa
\zeta(\frac{W_x}{x})\!+\!\frac{O_1}{x}\big)W_{xx}\!\!+\!O_2W_{xy}\!
+\!(\Ab \!+\!O_3)W_{yy}\!\\
&\qquad\qquad\qquad\qquad\qquad\qquad
-\!(2\!+\!O_4)W_x\!+\!O_5W_y\!\!
=\frac{O_1-xO_4}{\Aa}\qquad\inn\;\Qr_{\hat{\vR}, R}.
\end{split}
\end{equation}

\medskip
Step 3. For $z:=(x,y)\in \Qr_{r/2, R/2}$, define
\begin{equation}\label{parabRectangles-1}
R_{z}:=\big\{(s,t)\;\;:\;\; |s-x|<\frac{x}{8},\,
|t-y|<\frac{\sqrt{x}}{8}\big\}.
\end{equation}
Then
\begin{equation}\label{localizeRectangle}
R_{z}\subset \Qr_{r, 3R/4} \qquad\mbox{for any } z=(x,y)\in
\Qr_{r/2, R/2}.
\end{equation}

Fix $z_0=(x_0,y_0)\in \Qr_{r/2, R/2}$. Rescale $W$ in $R_{z_0}$ by
defining
\begin{equation}\label{parabRescaling-1}
W^{(z_0)}(S, T)=\frac{1}{x_0^{2+\alpha}}W(x_0+\frac{x_0}{8}S,
y_0+\frac{\sqrt{x_0}}{8}T) \qquad\mbox{for }\,\,(S, T)\in Q_1,
\end{equation}
where $Q_h=(-h,h)^2$ for $h>0$. Then, by (\ref{2-all alpha-growth}),
(\ref{3-all alpha-growth}), and (\ref{localizeRectangle}), we have
\begin{equation}\label{estInterLinfty-Rescaled-pf}
\|W^{(z_0)}\|_{C^0(\overline{Q_{1}})}\le \frac{1}{\Aa r^\alpha}.
\end{equation}
Moreover, since $W$ satisfies equation (\ref{c}), $W^{(z_0)}$
satisfies the following equation for $(S,T)\in Q_1$:
\begin{equation} \label{l}
\begin{split}
&(1+\frac S8)\Big(1+\Aa\zeta(\frac{8x_0^{\alp}W_S^{(z_0)}}{1+\frac
S8})+ \til O_1^{(z_0)}\Big)W_{SS}^{(z_0)}+\til
O_2^{(z_0)}W_{ST}^{(z_0)}
+(\Ab+\til O_3^{(z_0)})W_{TT}^{(z_0)}\\
& \qquad-\frac18(2+\til O_4^{(z_0)})W_S^{(z_0)}+\frac18\til
O_5^{(z_0)}W_T^{(z_0)}\!=\!\frac {(1+\frac S8)}{64\Aa x_0^{\alp}}
\big(\til O_1^{(z_0)}\!-\!\til O_4^{(z_0)}\big),
\end{split}
\end{equation}
where
\begin{equation*}
\begin{split}
&\til O_1^{(z_0)}(S,T)=\frac 1{x_0(1+\frac S8)}O_1(x,y), \qquad
\til O_2^{(z_0)}(S,T)=\frac 1{\sqrt{x_0}}O_2(x,y),\\
&\til O_3^{(z_0)}(S,T)=O_3(x,y), \qquad \til
O_4^{(z_0)}(S,T)=O_4(x,y), \qquad \til
O_5^{(z_0)}(S,T)=\sqrt{x_0}O_5(x,y),
\end{split}
\end{equation*}
with $x=x_0(1+\frac S8)$ and $y=y_0+\frac{\sqrt{x_0}}{8}T$. Then,
from (\ref{Oks-int})--(\ref{Oks-Der-int}), we find that, for all
$(S,T)\in \overline{Q_1}$ and $z_0\in  \Qr_{r/2, R/2}$ with $r\le
1$,
\begin{equation}
\label{rectangleCoefsEstimates}
\begin{split}
&|\til{O}_k^{(z_0)}(S,T)|\le 2N \sqrt{r}\qquad\text{for}\,\,\,  k=1,\dots, 5,\\
&|D\til{O}_k^{(z_0)}(S,T)|\le 2N\sqrt{r}\qquad\text{for}\,\,\,
k\ne 2,\\
&|D\til{O}_2^{(z_0)}(S,T)|\le 2N.
\end{split}
\end{equation}
%
%
Also, denoting the right-hand side of (\ref{l}) by $F^{(z_0)}(S,T)$,
we obtain from (\ref{rectangleCoefsEstimates}) that, for all $(S,
T)\in \overline{Q_1}$ and $z_0\in  \Qr_{r/2, R/2}$,
\begin{equation}
\label{rectangleRHSEstimates}
|F^{(z_0)}(S,T)|\le C r^{1-\alp}, \qquad |DF^{(z_0)}(S,T)|\le C
r^{\frac 12-\alp},
\end{equation}
where $C$ depends only on $N$ and $\Aa$.

Now, writing equation (\ref{l}) as
\begin{equation} \label{l-short}
\sum_{i,j=1}^2 A_{ij}(DW^{(z_0)}, S, T)\,D_{ij}^2W^{(z_0)}
+\sum_{i=1}^2 B_i(S, T)\,D_iW^{(z_0)} =F^{(z_0)} \qquad \inn\;\;
Q_1,
\end{equation}
we get from (\ref{defZeta}) and
(\ref{l})--(\ref{rectangleRHSEstimates}) that,
if $r>0$ is sufficiently small, depending only on the data, then
(\ref{l-short}) is uniformly elliptic with elliptic constants
depending only on $\Ab$ but independent of $z_0$, and that the
coefficients $A_{ij}(p, S, T)$, $B_i(S, T)$, and $F^{(z_0)}(S, T)$,
for $p\in\R^2$, $(S, T)\in Q_1$, satisfy
$$
\|A_{ij}\|_{C^1(\R^2\times \overline{Q_1}) }\le C, \quad \|(B_i,
\frac{F^{(z_0)}}{r^{1/2-\alp}})\|_{C^1( \overline{Q_1}) }\le C,
$$
where $C$ depends only on the data and is independent of  $z_0$.
Then, by \cite[Theorem A1]{ChenFeldman} and
(\ref{estInterLinfty-Rescaled-pf}),
\begin{equation} \label{elliptic_est}
\|W^{(z_0)}\|_{C^{2,\alp}(\overline{Q_{1/2}})}\le
C\big(\|W^{(z_0)}\|_{C^0(\overline{Q_1})}+
\|F^{(z_0)}\|_{C^\alp(\overline{Q_1})}\big) \le
C(\frac{1}{r^{\alp}a}+ r^{1/2-\alp})=:\hat{C},
\end{equation}
where $C$ depends only on the data and $\alp$ in this case. {}From
(\ref{elliptic_est}),
\begin{equation} \label{q}
|D_x^iD_y^j W(x_0,y_0)|\le C x_0^{2+\alp-i-j/2}\qquad \text{for
all}\quad (x_0, y_0)\in  \Qr_{r/2, R/2}, \quad 0\le i+j\le 2.
\end{equation}

\medskip
Step 4. It remains to prove the $C^{\alp}$-continuity of $D^2W$ in
$\overline{\Qr_{r/2, R/2}}$.

For two distinct points $z_1=(x_1,y_1),z_2=(x_2,y_2)\in \Qr_{r/2,
R/2}$, consider
$$
A:=\frac{|W_{xx}(z_1)-W_{xx}(z_2)|}{|z_1-z_2|^{\alp}}.
$$
Without loss of generality, assume that $x_1\le x_2$. There are two
cases:

{\em Case 1.} $z_1\in R_{z_2}$. Then
$$
x_1=x_2+\frac{x_2}{8}S, \quad y_1=y_2+\frac{\sqrt{x_2}}{8}T \qquad
\text{for some } (S,T)\in Q_1.
$$
By
(\ref{elliptic_est}),
$$
\frac{|W_{SS}^{(z_2)}(S,T)-W_{SS}^{(z_2)}(0,0)|}{(S^2+T^2)^{\alp/{2}}}\le
\hat{C},
$$
which is
$$
\frac{|W_{xx}(x_1,y_1)-W_{xx}(x_2,y_2)|}{((x_1-x_2)^2+x_2(y_1-y_2)^2)^{\alp/{2}}}
\le \hat{C}.
$$
Since $x_2\in (0, r)$ and $r\le 1$, the last estimate implies
$$
\frac{|W_{xx}(x_1,y_1)-W_{xx}(x_2,y_2)|}{((x_1-x_2)^2+(y_1-y_2)^2)^{\alp/{2}}}\le
\hat{C}.
$$

{\em Case 2.} $z_1\notin R_{z_2}$. Then, either
$|x_1-x_2|>\frac{x_2}{8}$ or $|y_1-y_2|>\frac{\sqrt{x_2}}{8}$. Since
$0\le x_2\le r\le 1$, we find
$$
|z_1-z_2|^\alp\ge \Big(\frac{x_2}{8}\Big)^\alp.
$$
Thus, using (\ref{q}) and $x_1\le x_2$, we obtain
$$
\frac{|W_{xx}(z_1)-W_{xx}(z_2)|}{|z_1-z_2|^{\alp}}\le
\frac{|W_{xx}(z_1)|+|W_{xx}(z_2)|}{|z_1-z_2|^{\alp}}\le
\hat{C}\frac{x_1^\alp+x_2^\alp}{x_2^\alp}\le 2 \hat{C}.
$$

Therefore, $A\le 2\hat{C}$ in both cases, where $\hat{C}$ depends on
$\alp, r$, and the data. Since $z_1\ne z_2$ are arbitrary points of
$\Qr_{r/2, R/2}$, we obtain
\begin{equation} \label{est_1}
[W_{xx}]_{C^\alp(\overline{\Qr_{r/2, R/2}})}\le 2 \hat{C}.
\end{equation}
The estimates for $W_{xy}$ and $W_{yy}$ can be obtained similarly.
In fact, for these derivatives, we obtain the stronger estimates:
For any $\delta\in (0, r/2]$,
$$
[W_{xy}]_{C^\alp(\overline{\Qr_{\delta, R/2}})}\le
\hat{C}\sqrt{\delta},\qquad [W_{yy}]_{C^\alp(\overline{\Qr_{\delta,
R/2}})}\le \hat{C}\delta,
$$
where $\hat{C}$ depends on $\alp, r$, and the data, but is
independent of $\delta>0$ and $z_0$.

Thus, $W\in C^{2,\alpha}(\overline{\Qr_{r, R/2}})$ with
$\|W\|_{C^{2,\alpha}(\overline{\Qr_{r, R/2}})}$ depending only on
the data because $r>0$ depends on the data. Moreover,  (\ref{q})
implies $D^2 W(0, y)= 0$ for any $|y|\le R/2$. This concludes the
proof of Theorem \ref{mainRegularityThm}. $\Box$

\section{Optimal Regularity of Solutions to Regular Shock Reflection
across the Sonic Circle} \label{OptReg-RegRefl-section}

\smallskip
As we indicated in Section 2,
the global solution $\varphi$ constructed in \cite{ChenFeldman} is
at least $C^{1,1}$ near the sonic circle $\PtUpL \PtUpR$. On the
other hand, the behavior of solutions to regular shock reflection
has not been understood completely; so it is essential to understand
first the regularity of regular reflection solutions.
%
In this section, we prove that $C^{1,1}$ is in fact the optimal
regularity of any solution $\varphi$ across the sonic circle $\PtUpL
\PtUpR$ in the class of standard regular reflection solutions. Our
main results include the following three ingredients:

(i) There is no a regular reflection solution that is $C^2$ across
the sonic circle;

(ii) For the solutions constructed in \cite{ChenFeldman} or, more
generally, for any regular reflection solution satisfying properties
(\ref{phi-GE-phi2}) and (\ref{C11NormEstimate})--(\ref{FBnearSonic})
at the end of Section 2,
%
%
$\varphi$ is $C^{2,\alpha}$ in the subsonic region $\Omega$ up to
the sonic circle $\PtUpL\PtUpR$,
excluding the endpoint $\PtUpL$, but $D^2\varphi$ has a jump across
$\PtUpL\PtUpR$;

(iii)
In addition, $D^2\varphi$ does not have a limit at $\PtUpL$ from
$\Omega$.

\smallskip
In order to state these results, we first define
the class of
regular reflection solutions. As proved in \cite{ChenFeldman}, when
the wedge angle $\theta_w$ is large, such a regular reflection
configuration exists; and in \cite{ChenFeldman3}, we extend this to
other wedge-angles for which regular reflection configuration
exists.
%
%

\smallskip Now we define the class of regular
reflection solutions.

\begin{definition}\label{RegReflSolDef}
Let $\gamma > 1$, $\rho_1>\rho_0>0$, and $u_1>0$ be constants, and
let $\xi_0$ be defined by {\rm (\ref{shocklocation})}.
Let the incident shock $S=\{\xi=\xi_0\}$ hit the wedge at the point
$P_0=(\xi_0, \xi_0 \tan\theta_w)$, and let state $(0)$ and state
$(1)$ ahead of and behind $S$ be given by
\eqref{flatOrthSelfSimShock1} and \eqref{flatOrthSelfSimShock2},
respectively.
The function $\vphi\in C^{0,1}(\Lambda)\cap C^2(\Omega)$ is a
regular reflection solution if $\vphi$ is a solution to Problem $2$
satisfying {\rm (\ref{phi-states-0-1-2})} and such that

{\rm (a)} there exists state $(2)$ of  the form {\rm
(\ref{state2a})} with $u_2>0$, satisfying the entropy condition
$\rho_2>\rho_1$ and the Rankine-Hugoniot condition $(\rho_1
D\varphi_1-\rho_2 D\varphi_2)\cdot\nu=0$
along the line $S_1:=\{\varphi_1=\varphi_2\}$ which contains the points
$P_0$ and $\PtUpL$, such that $\PtUpL\in \Lambda$ is on the sonic
circle of state $(2)$, and state $(2)$ is supersonic along
$P_0\PtUpL$;

{\rm (b)}
equation \eqref{1.1.5} is elliptic in $\Omega$;

{\rm (c)} $\varphi\ge \varphi_2$ on the part $P_1P_2=\Gamma_{shock}$
of the reflected shock.
%
\end{definition}

\begin{remark}
If state $(2)$ exists and supersonic, then the line
$S_1=\{\varphi_1=\varphi_2\}$ necessarily intersects the sonic
circle of state $(2)$; see the argument in \cite{ChenFeldman}
starting from {\rm (3.5)} there. Thus the only assumption regarding
the point $\PtUpL$ is that $S_1$ intersects the sonic circle within
$\Lambda$.
\end{remark}

\begin{remark}
The global solution constructed in \cite{ChenFeldman} is a regular
reflection solution, which is a part of the assertions at the end of
Section {\rm 2}.
\end{remark}

\begin{remark}
There may exist a global regular reflection configuration when state
$(2)$ is subsonic which is a very narrow regime \cite{CF, Neumann}.
Such a case does not involve the difficulty of elliptic degeneracy
which we are facing in our case.
\end{remark}

\begin{remark}
Since $\varphi=\varphi_1$ on $\Shock$ by \eqref{phi-states-0-1-2},
Condition {\rm (c)} in Definition {\rm \ref{RegReflSolDef}} is
equivalent to
$$
\Gamma_{shock}\subset \{\varphi_2\le \varphi_1\},
$$
that is, $\Gamma_{shock}$ is below $S_1$.
\end{remark}

Furthermore, we have
\begin{lemma}
For any regular reflection solution $\varphi$ in the sense of
Definition {\rm \ref{RegReflSolDef}},
\begin{equation}\label{varphi-varphi-2}
\varphi > \varphi_2 \qquad\mbox{in}\,\, \Omega.
\end{equation}
\end{lemma}

\begin{proof}
By \eqref{1.1.5}--\eqref{1.1.6} and \eqref{state2a},
$\psi:=\varphi-\varphi_2$ satisfies
\begin{equation}\label{eq-psi}
\begin{split}
&(\til c^2-(\psi_{\xi}-\xi+u_2)^2)\psi_{\xi\xi}
+(\til c^2-(\psi_{\eta}-\eta+u_2\tan\theta_w)^2)\psi_{\eta\eta}\\
&\phantom{aaaaaaaaaaa}-2(\psi_{\xi}-\xi+u_2)(\psi_{\eta}-\eta+u_2\tan\theta_w)
\psi_{\xi\eta}=0\;\;\qquad \text{in}\;\;\Omega,
\end{split}
\end{equation}
where $\til c^2(D\psi,\psi,\xi,\eta)
=c_2^2+(\gam-1)\Bigl((\xi-u_2)\psi_{\xi}+(\eta-u_2\tan\theta_w)\psi_{\eta}
-\frac 12|D\psi|^2-\psi\Bigr)$.

Since the equation (\ref{eq-psi}) is elliptic inside $\Omega$ and
$\varphi$ is smooth inside  $\Omega$, it follows that (\ref{eq-psi})
is uniformly elliptic in any compact subset of $\Omega$.
Furthermore, we have
\begin{align*}
\psi=0 \qquad&\text{on}\;\;\Gamma_{sonic},\\
D\psi\cdot (-\sin\theta_w,\cos\theta_w)=0 \qquad &\text{on}\;\;\Gamma_{wedge},\\
\psi_{\eta}=-u_2\tan\theta_w<0 \qquad &\text{on}\;\;\partial\Omega\cap \{\eta=0\},\\
\psi\ge 0 \qquad &\text{on}\;\;\Gamma_{shock}\;\quad \text{(by
Definition 4.1 (c))}.
\end{align*}
Then the strong maximum principle implies
$$
\psi>0\;\;\qquad \text{in}\;\;\Omega,
$$
which is \eqref{varphi-varphi-2}. This completes the proof.
\end{proof}

Now we first show that any regular reflection solution in our case
cannot be $C^2$ across the sonic circle  $\Sonic:=\PtUpL \PtUpR$.

\begin{theorem} \label{noC2acrossSonicLineThm}
Let $\vphi$ be a regular reflection solution in the sense of
Definition {\rm 4.1}. Then $\vphi$ cannot be $C^2$ across the sonic
circle $\Sonic$.
\end{theorem}

\begin{proof} On the contrary, assume that $\vphi$ is $C^2$ across  $\Sonic$.
Then  $\psi=\vphi-\varphi_2$ is also $C^2$  across  $\Sonic$, where
$\vphi_2$ is given by (\ref{state2a}). Moreover, since $\psi\equiv
0$ in $P_0\PtUpL\PtUpR$ by (\ref{phi-states-0-1-2}), we have
$D^2\psi(\xi,\eta)=0$ at all $(\xi, \eta)\in \Sonic$.

Now substituting $\vphi=\psi+\vphi_2$ into equation \eqref{1.1.5}
 and writing the
resulting equation in the $(x,y)$-coordinates (\ref{coordNearSonic})
in the domain $\Omega_{\eps_0}$ defined by (\ref{OmegaXY}), we find
by an explicit calculation that $\psi(x,y)$ satisfies equation
(\ref{mainEq-allTerms}) in $\Omega_{\eps_0}$ with $a=\gam+1$ and
$b=\frac{1}{c_2}$ and with $O_j=O_j(x, y, \psi, \psi_x, \psi_y),
j=1, \dots, 5,$ given by
\begin{equation}\label{OtermsRegRefl}
\begin{split}
&O_1(\grad\psi,\psi,x) = -\frac{x^2}{c_2}+{\gamma+1\over
2c_2}(2x-\psi_x)\psi_x -{\gamma-1\over c_2}\big(\psi+{1\over
2(c_2-x)^2}\psi_y^2\big),
\\
&O_2(\grad\psi,\psi, x)=-{2\over c_2(c_2-x)^2}(\psi_x+c_2-x)\psi_y,
\\
&O_3(\grad\psi,\psi, x)= {1\over c_2(c_2-x)^2}\Big(x(2c_2-x)-
(\gamma-1)(\psi+(c_2-x)\psi_x+{1\over 2}\psi_x^2)
\\
& \qquad\qquad\qquad\qquad\qquad\qquad\quad
-\frac{\gamma+1}{2(c_2-x)^2}\psi_y^2\Big),
\\
&O_4(\grad\psi,\psi, x) =\frac{1}{c_2-x}\Big(x- {\gamma-1\over
c_2}\big(\psi+(c_2-x)\psi_x+{1\over 2}\psi_x^2 +\frac{\psi_y^2}{2
(c_2-x)^2}\big)\Big),
\\
&O_5(\grad\psi,\psi, x) =
-\frac{1}{c_2(c_2-x)^3}\big(\psi_x+2c_2-2x\big)\psi_y.
\end{split}
\end{equation}

Let $(0, y_0)$ be a point in the relative interior of $\Sonic$. Then
$(0, y_0)+\Qr_{\vR, R}\subset \Om_{\eps_0}$ if $\vR, R>0$ are
sufficiently small. By shifting the coordinates $(x,y)\to (x,
y-y_0)$, we can assume $(0, y_0)=(0,0)$ and $\Qr_{\vR, R}\subset
\Om_{\eps_0}$. Note that the shifting coordinates in the
$y$-direction does not change the expressions in
(\ref{OtermsRegRefl}).

Since $\psi\in C^2(\Om_{\eps_0}\cup\Sonic)$ with $D^2\psi\equiv 0$
on $\Sonic$, reducing $\vR$ if necessary, we get $|D\psi|\le \delta
x$ in $\Qr_{\vR, R}$, where $\delta>0$ is so small
 that (\ref{x-deriv-small}) holds
in $\Om_{\eps_0}$, with $\beta=M=1$, and that the terms $O_i$
defined by (\ref{OtermsRegRefl}) satisfy
(\ref{Oks-int})--(\ref{Oks-Der-int}) with $M=1$. Also, from
Definition \ref{RegReflSolDef}, we obtain that
$\psi=\vphi-\vphi_2>0$ in $\Qr_{\vR, R}$. Now we can apply
Proposition \ref{h} to conclude
$$
\psi(x,y)\ge \mu x^2\;\qquad\text{on}\;\;\Qr_{r, 15R/16}
$$
for some $\mu, r>0$. This contradicts the fact that $D^2\psi(0, y)=
0$ for all $y\in (-R, R)$, that is, $D^2\psi(\xi,\eta)=0$ at any
$(\xi, \eta)\in \Sonic$.
\end{proof}

\medskip
In the following theorem, we study more detailed regularity of
$\psi$ near the sonic circle in the case of  $C^{1,1}$ regular
reflection solutions. Note that this class of solutions especially
includes the solutions constructed in \cite{ChenFeldman}.


\begin{theorem} \label{mainthm1}
Let
$\vphi$ be a regular reflection solution in the sense of Definition
{\rm \ref{RegReflSolDef}} and satisfy the properties:
\begin{enumerate}
\item[(a)]\label{C11NormEstimate-a}
$\varphi$ is $C^{1,1}$ across the part $\Sonic$ of the sonic circle,
i.e., there exists $\eps_0>0$ such that $\vphi\in
C^{1,1}(\overline{P_0\PtUpL\PtLwL\PtLwR}\cap\{c_2-\eps_0<r<c_2+\eps_0\})$;
\item[(b)]\label{psi-x-est-a}
there exists  $\delta_0>0$ so that, in the coordinates {\rm
(\ref{coordNearSonic})},
\begin{equation}\label{ell-a}
|\partial_x(\vphi-\vphi_2)(x,y)|\le
\frac{2-\delta_0}{\gam+1}x\;\qquad\inn \;\;\Om_{\eps_0};
\end{equation}
\item[(c)]\label{FBnearSonic-a}
there exist $\omega>0$ and a function $y=\hat{f}(x)$ such that, in
the coordinates {\rm (\ref{coordNearSonic})},
\begin{equation}\label{OmegaXY-a}
\begin{split}
&\Om_{\eps_0}=\{(x,y)\,:\, x\in(0,\;\eps_0),\;\; 0< y<\hat{f}(x)\}, \\
&\Shock\cap\{0\le x\le \eps_0\}=\{(x,y)\,:\,x\in(0,\;\eps_0),\;\;
y=\hat{f}(x)\},
\end{split}
\end{equation}
and
\begin{equation} \label{fhat-a}
\|\hat{f}\|_{C^{1,1}([0,\;\eps_0])}<\infty\;\;,\;\;
\frac{d\hat{f}}{dx}\!\ge\omega>0\!\;\;\; \text{for}\;\;0<x<\eps_0.
\end{equation}
\end{enumerate}
Then we have
\begin{enumerate}
\renewcommand{\theenumi}{\roman{enumi}}
\item\label{C2alpSonic}
$\vphi$ is $C^{2,\alp}$ up to $\Sonic$ away from the point $\PtUpL$
for any $\alp\in (0,1)$.
That is, for any $\alpha\in (0,1)$ and any given
$(\xi_0,\eta_0)\in\overline\Sonic\setminus\{\PtUpL\}$, there exists
$K<\infty$ depending only on $\rho_0,\,\rho_1,\,\gam,\,\eps_0,\,
\alpha, \|\vphi\|_{C^{1,1}(\Om_{\eps_0})}$, and
$d=dist((\xi_0,\eta_0),\;\Shock)$ so that
\begin{equation*}
\|\vphi\|_{2,\alp;\overline{B_{d/2}(\xi_0,\eta_0)\cap
\Om_{\eps_0/2}}}\le K;
\end{equation*}
%
\item\label{limitSonic}
For any $(\xi_0,\eta_0)\in \Sonic\setminus\{\PtUpL\}$,
\begin{equation*}
\lim_{(\xi,\eta)\to (\xi_0,\eta_0)\atop
 (\xi,\eta)\in\Om }(D_{rr}\vphi-D_{rr}\vphi_2)=\frac{1}{\gamma+1};
\end{equation*}

\item\label{2ndderJumpSonic}
$D^2\vphi$ has a jump across $\Sonic$: For any $(\xi_0,\eta_0)\in
\Sonic\setminus\{\PtUpL\}$,
\begin{equation*}
\lim_{(\xi,\eta)\to (\xi_0,\eta_0)\atop
 (\xi,\eta)\in\Om }D_{rr}\vphi \;-\;
 \lim_{(\xi,\eta)\to (\xi_0,\eta_0)\atop
 (\xi,\eta)\in\Lambda\setminus\Om }D_{rr}\vphi\;=\;
 \frac{1}{\gamma+1};
\end{equation*}
\item\label{noLimAtP1}
The limit $\lim_{(\xi,\eta)\to\PtUpL \atop (\xi,\eta)\in \Om}
D^2\vphi$ does not exist.
\end{enumerate}
\end{theorem}

\begin{proof}
The proof consists of
seven steps.

Step 1. Let
$$
\psi:=\vphi-\vphi_2.
$$
By \eqref{phi-states-0-1-2} and (\ref{ell-a}),
we have
\begin{equation}\label{psiZeroOnSonic}
\psi(0, y)=\psi_x(0, y)
=\psi_y(0, y)=0\qquad\text{for all }\; (0, y)\in\Sonic,
\end{equation}
and thus, using also (\ref{OmegaXY-a})--\eqref{fhat-a}, we find that
\begin{equation}\label{psiDerivEstNearSonic}
|\psi(x,y)|\le Cx^2,\quad
|D_{x,y}\psi(x,y)|\le Cx\qquad
\text{for all }\;(x,y)\in \Om_{\eps_0},
\end{equation}
where $C$ depends only on $\|\psi\|_{C^{1,1}(\overline{
\Om_{\eps_0}})}$ and $\|\hat{f}\|_{C^{1}([0,\; \eps_0])}$.

Recall that, in the $(x,y)$-coordinates (\ref{coordNearSonic})
 in the domain $\Om_{\eps_0}$ defined by
(\ref{OmegaXY-a}), $\psi(x,y)$ satisfies equation
(\ref{mainEq-allTerms}) with $O_i=O_i(x, y, \psi, \psi_x, \psi_y)$
given by (\ref{OtermsRegRefl}). Then  it follows from
(\ref{OtermsRegRefl}) and (\ref{psiDerivEstNearSonic}) that
\eqref{Oks-int}--\eqref{Oks-Der-int} hold with $N$ depending only on
$\eps_0$, $\|\psi\|_{C^{1,1}(\overline{ \Om_{\eps_0}})}$, and
$\|\hat{f}\|_{C^{1}([0,\; \eps_0])}$.

\medskip
Step 2. Now, using (\ref{ell-a}) and reducing $\eps_0$ if necessary,
we conclude that
 (\ref{mainEq-allTerms}) is uniformly elliptic on $\Om_{\eps_0}\cap\{x>\delta\}$
for any $\delta\in (0, \eps_0)$. Moreover, by (c), equation
(\ref{mainEq-allTerms}) with (\ref{OtermsRegRefl}), considered as a
linear elliptic equation, has $C^1$ coefficients. Furthermore, since
the boundary conditions (\ref{boundary-condition-3}) hold for
$\vphi$ and $\vphi_2$, especially on $\Wedge=\{y=0\}$, it follows
that, in the $(x, y)$-coordinates,
we have
\begin{equation}\label{boundary-condition-wedge-xy}
\psi_y(x, 0) =0  \,\,\qquad\text{for all }\; x\in (0, \eps_0).
\end{equation}
Then, by the standard regularity theory for the oblique derivative
problem for linear, uniformly elliptic equations, $\psi$ is $C^2$ in
$\Om_{\eps_0}$ up to $\partial\Om_{\eps_0}\cap \{0<x<\eps_0,
\;y=0\}$. {}From this and (c), we have
\begin{equation}\label{regPsi}
\psi\in C^{1,1}(\overline{\Om_{\eps_0}})\cap C^2(\Om_{\eps_0}\cup\Wedge^{(\eps_0)}),
\end{equation}
where
$\Wedge^{(\eps_0)}:=\Wedge\cap\{0<x<\eps_0\}\equiv\{(x,0)\;:\;0<x<\eps_0\}$.

Reflect $\Om_{\eps_0}$ with respect to the $y$-axis, i.e., using
(\ref{OmegaXY-a}), define
\begin{equation}\label{extendedOmega}
\hat\Om_{\eps_0}:=\{(x,y):x\in(0,\;\eps_0),\;\; -\hat{f}(x) <
y<\hat{f}(x)\}.
\end{equation}
Extend $\psi(x,y)$ from $\Om_{\eps_0}$ to $\hat\Om_{\eps_0}$ by the
even reflection, i.e., defining $\psi(x, -y)=\psi(x, y)$ for
$(x,y)\in \Om_{\eps_0}$. Using
(\ref{boundary-condition-wedge-xy})--(\ref{regPsi}), we conclude
that the extended function $\psi(x,y)$ satisfies
\begin{equation}\label{regExtendedPsi}
\psi\in C^{1,1}(\overline{\hat\Om_{\eps_0}})\cap C^2(\hat\Om_{\eps_0}).
\end{equation}

Now we use the explicit expressions \eqref{mainEq-allTerms} and
(\ref{OtermsRegRefl}) to find that, if $\psi(x,y)$ satisfies
equation (\ref{mainEq-allTerms}) with (\ref{OtermsRegRefl}) in
$\Om_{\eps_0}$, then the function $\tilde\psi(x,y):=\psi(x, -y)$
also satisfies (\ref{mainEq-allTerms}) with
$O_k(\grad\tilde\psi,\tilde\psi, x)$ defined by
(\ref{OtermsRegRefl})
 in $\Om_{\eps_0}$.
Thus, in  the extended domain
$\hat\Om_{\eps_0}$,
the extended $\psi(x,y)$ satisfies  (\ref{mainEq-allTerms})
 with $O_1, \dots, O_5$ defined by the expressions (\ref{OtermsRegRefl})
in $\hat\Om_{\eps_0}$.

Moreover, by (\ref{phi-states-0-1-2}), it follows that $\psi=0$ on
$\Sonic$. Thus, in the $(x,y)$-coordinates, for the extended $\psi$,
we obtain
\begin{equation}\label{boundary-condition-sonic-xy}
\psi(0, y) =0  \qquad\text{for all }\; y\in (-\hat f(0), \hat f(0)).
\end{equation}
Also, using $\varphi\ge \varphi_2$ in $\Omega$,
\begin{equation}\label{psiGE0-xy}
\psi(0, y) \ge 0  \qquad\text{in }\; \hat\Om_{\eps_0}.
\end{equation}

Step 3. Let $P=(\xi_*, \eta_*)\in\Sonic\setminus \{\PtUpL\}$. Then,
in the $(x,y)$-coordinates,  $P=(0, y_*)$ with $y_*\in [0,\hat{f}(0)
)$. Then, by (\ref{fhat-a}) and (\ref{extendedOmega}), there exist
$\vR, R>0$, depending only on $\eps_0, c_2=c_2(\rho_0, \rho_1, u_1,
\theta_w)$, and $d=\dist((\xi_*, \eta_*), \Shock)$, such that
$$
(0, y_*)+\Qr_{\vR, R}\subset \hat\Om_{\eps_0}.
$$
Then, in $\Qr_{\vR, R}$, the function $\hat\psi(x, y):=\psi(x,
y-y_*)$ satisfies all the conditions of Theorem
\ref{mainRegularityThm}. Thus, applying Theorem
\ref{mainRegularityThm} and expressing the results in terms of
$\psi$, we obtain that, for all $y_*\in [0,\hat{f}(0))$,
\begin{equation}\label{2ndDerivRegRefl}
\lim_{(x,y)\to (0,y_*)\atop (x,y)\in\Omega}
\psi_{xx}(x,y)=\frac{1}{\gamma+1},\qquad \lim_{(x,y)\to (0,y_*)\atop
(x,y)\in\Omega}\psi_{xy}(x,y) = \lim_{(x,y)\to (0,y_*)\atop
(x,y)\in\Omega}\psi_{yy}(x,y)=0.
\end{equation}
Since $\psi_{rr}=\psi_{xx}$ by (\ref{coordNearSonic}), this implies
assertions (\ref{C2alpSonic})--(\ref{limitSonic}) of Theorem
\ref{mainthm1}.

Now assertion (\ref{2ndderJumpSonic}) of Theorem \ref{mainthm1}
follows from  (\ref{limitSonic}) since, by
(\ref{phi-states-0-1-2}), $\vphi=\vphi_2$ in $B_\eps(\xi_*,
\eta_*)\setminus\Omega$ for small $\eps>0$ and $\vphi_2$ is a
$C^\infty$-smooth function in $\R^2$.

\medskip
Step 4. It remains to show assertion (\ref{noLimAtP1}) of Theorem
\ref{mainthm1}. We prove this by contradiction. Assume that
assertion (\ref{noLimAtP1}) is false, i.e., there exists a limit of
$D^2\psi$ at $\PtUpL$ from $\Om$. Then our strategy is to choose two
different sequences of points converging to $\PtUpL$ and show that
the limits of $\psi_{xx}$ along the two sequences are different,
which reaches to a contradiction.
We note that, in the $(x,y)$-coordinates, the point
$\PtUpL=(0,\;\hat{f}(0))$.

\medskip
Step 5. {\bf A sequence close to $\Sonic$}. Let
$\{y_m^{(1)}\}_{m=1}^{\infty}$ be a sequence such that $y_m^{(1)}
\in(0,\;\hat{f}(0))$ and $\lim_{m\rightarrow
\infty}y_m^{(1)}=\hat{f}(0)$. By (\ref{2ndDerivRegRefl}), there
exists $x_m^{(1)}\in(0,\;\frac 1m)$ such that
$$
\big|\psi_{xx}(x_m^{(1)},y_m^{(1)})-\frac{1}{\gamma+1}\big|+
|\psi_{xy}(x_m^{(1)},y_m^{(1)})| +
|\psi_{yy}(x_m^{(1)},y_m^{(1)})|<\frac 1m
$$
 for each
$m=1,2,3,\dots$. Moreover, using (\ref{fhat-a}), we have
$$
y_m^{(1)}< \hat{f}(0)\le \hat{f}(x_m^{(1)}).
$$
Thus, using (\ref{OmegaXY-a}), we have
\begin{equation}\label{seqCloseSonic}
\begin{split}
&(x_m^{(1)}, y_m^{(1)})\in\Om,\;\qquad
\lim_{m\rightarrow\infty}(x_m^{(1)},y_m^{(1)})=(0,\hat{f}(0)),\;\;
\\
&\lim_{m\rightarrow \infty}\psi_{xx}(x_m^{(1)},y_m^{(1)})=\frac{1}{\gamma+1},
\qquad
\lim_{m\rightarrow \infty}\psi_{xy}(x_m^{(1)},y_m^{(1)})
=\lim_{m\rightarrow \infty}\psi_{yy}(x_m^{(1)},y_m^{(1)})
=0.
\end{split}
\end{equation}

\medskip
Step 6.  {\bf The Rankine-Hugoniot conditions on $\Shock$}.
In order to construct another sequence, we first combine the
Rankine-Hugoniot conditions on $\Shock$ into a condition of the
following form:

\begin{lemma}\label{bdryCondShock-xy}
There exists $\eps\in (0, \eps_0)$ such that $\psi$ satisfies
\begin{equation}\label{combinedCondOnShock}
\hat{b}_1(x,y)\psi_x+\hat{b}_2(x,y)\psi_y+\hat{b}_3(x,y)\psi=0\qquad
\text{on}\;\;\Shock\cap\{0<x<\eps\},
\end{equation}
where $\hat{b}_k\in C(\overline{\Shock\cap\{0\le x\le\eps \}})$ and
further satisfies
\begin{equation}\label{coeffs-combinedCondOnShock}
\hat{b}_1(x,y)\ge \lambda, \quad |\hat{b}_2(x,y)|\le \frac 1\lambda,
\quad |\hat{b}_3(x,y)|\le \frac 1\lambda\;\qquad
\text{on}\;\;\Shock\cap\{0<x<\eps \}
\end{equation}
for some constant $\lambda>0$.
\end{lemma}

\begin{proof}
To prove this, we first work in the $(\xi, \eta)$-coordinates. Since
$$
\vphi=\vphi_1,\quad \rho \grad\vphi\cdot\nu=\rho_1
\grad\vphi_1\cdot\nu \qquad\,\,\text{on}\,\,\Shock,
$$
then $\nu$ is parallel to $\grad\vphi_1-\grad\vphi$ so that
\begin{equation}\label{combinedRH}
(\rho_1 \grad\vphi_1-\rho\grad\vphi)\cdot(\grad\vphi_1-\grad\vphi) =
0 \qquad\,\,\text{on}\,\,\Shock.
\end{equation}
Since both $\vphi$ and $\vphi_2$ satisfy
(\ref{1.1.6})--(\ref{c-through-density-function}) and
$\psi:=\vphi-\vphi_2$, we have
\begin{equation}\label{rho-c-psi}
\begin{split}
&\rho=\rho(\grad\psi,\psi,\xi,\eta) =\Big(\rho_2^{\gamma-1}
+(\gamma-1)\big((\xi-u_2)\psi_\xi+(\eta-v_2)\psi_\eta-\frac{1}{2}|\grad\psi|^2-\psi
\big)\Big)^\frac{1}{\gamma-1},\\
&c^2=c^2(\grad\psi,\psi,\xi,\eta)
=c_2^2+(\gamma-1)\Big((\xi-u_2)\psi_\xi+(\eta-v_2)\psi_\eta-\frac{1}{2}|\grad\psi|^2-\psi
\Big).
\end{split}
\end{equation}
Then, writing $\vphi=\vphi_2+\psi$ and using
\eqref{flatOrthSelfSimShock1}--\eqref{flatOrthSelfSimShock2}, we
rewrite (\ref{combinedRH}) as
\begin{equation}\label{combinedRH-1}
E(\psi_\xi, \psi_\eta,\psi, \xi, \eta) = 0
\qquad\,\,\text{on}\,\,\Shock,
\end{equation}
where, for $(p_1, p_2, p_3,  \xi, \eta)\in\R^5 $,
\begin{eqnarray}
&&E(p_1,p_2, p_3, \xi, \eta)\label{combinedRH-1-1}\\
&&\quad =\rho_1\Big((u_1-\xi)(u_1-u_2-p_1)+\eta(v_2+p_2)\Big)\nonumber\\
\nonumber
&&\qquad -\rho(p_1,p_2, p_3, \xi, \eta)\Big((u_2-\xi+p_1)(u_1-u_2-p_1)
-(v_2-\eta+p_2)(v_2+p_2)\Big),\\
&&\rho(p_1,p_2, p_3,  \xi, \eta)=\Big(\rho_2^{\gamma-1}
+(\gamma-1)\big((\xi-u_2)p_1+(\eta-v_2)
p_2-\frac{1}{2}|p|^2-p_3\big) \Big)^\frac{1}{\gamma-1}, \quad
\label{combinedRH-1-2}
\end{eqnarray}
with  $v_2:= u_2\tan\theta_w$.

Since both points $P_0$ and $\PtUpL$ lie on
$S_1=\{\vphi_1=\vphi_2\}$, we have
$$
(u_1-u_2)(\xi_1-\xi_0)-v_2(\eta_1-\eta_0)=0,
$$
where $(\xi_1,\eta_1)$ are the coordinates of $\PtUpL$. Now, using
the condition $\vphi=\vphi_1$ on $\Shock$, i.e.,
$\psi+\vphi_2=\vphi_1$ on $\Shock$, we have
\begin{equation}\label{etaOnShock}
\eta=\frac{(u_1-u_2)(\xi-\xi_1)-\psi(\xi,
\eta)}{v_2}+\eta_1\qquad\,\, \text{on}\,\,\Shock.
\end{equation}
{}From (\ref{combinedRH-1}) and (\ref{etaOnShock}), we conclude
\begin{equation}\label{combinedRH-2}
F(\psi_\xi, \psi_\eta,\psi, \xi) = 0\; \qquad\,\text{on}\,\,\Shock,
\end{equation}
where
\begin{equation}\label{definition of F}
F(p_1, p_2, p_3,  \xi)=E(p_1, p_2, p_3, \xi,
\frac{(u_1-u_2)(\xi-\xi_1)-p_3}{v_2}+\eta_1).
\end{equation}

Now, from (\ref{combinedRH-1-1})--(\ref{combinedRH-1-2}), we obtain
that, for any $\xi\in \R$,
\begin{equation}\label{combinedRH-zero}
\begin{split}
F(0,0,0,  \xi)&=E(0, 0, 0, \xi, \frac{(u_1-u_2)(\xi-\xi_1)}{v_2}+\eta_1)\\
&=\rho_1\Big((u_1-\xi_1)(u_1-u_2)+v_2\eta_1\Big)
 -\rho_2\Big((u_2-\xi_1)(u_1-u_2)-v_2(v_2-\eta_1)\Big)\\
&=\Big(\rho_1\grad\vphi_1(\xi_1,\eta_1)-
\rho_2\grad\vphi_2(\xi_1,\eta_1)
\Big)\cdot (u_1-u_2,\;-v_2)\\
&=0,
\end{split}
\end{equation}
where the last expression is zero since it represents the right-hand
side of the Rankine-Hugoniot condition (\ref{FBConditionSelfSim-0})
at the point $\PtUpL$ of the shock $S_1=\{\vphi_1=\vphi_2\}$
separating state (2) from state (1).

Now we write condition (\ref{combinedRH-2}) in the
$(x,y)$-coordinates on $\Shock\cap\{0<x<\eps_0 \}$. By
(\ref{coordPolar})--(\ref{coordNearSonic}) and (\ref{combinedRH-2}),
we have
\begin{equation}\label{combinedRH-xy}
\Psi(\psi_x,\psi_y, \psi, x,y) = 0
\qquad\,\,\text{on}\,\,\Shock\cap\{0<x<\eps_0 \},
\end{equation}
where
\begin{equation}\label{combinedRH-xy-1}
\begin{split}
\Psi(p_1, p_2, p_3,  x, y)=& F\big(-p_1\cos(y+\theta_w)
-\frac{p_2}{c_2-x}\sin(y+\theta_w),\; \\
&\,\,\,-p_1\sin(y+\theta_w) +\frac{p_2}{c_2-x}\cos(y+\theta_w),\;
p_3,\;u_2+(c_2-x)\cos(y+\theta_w)\big).
\end{split}
\end{equation}

{}From  (\ref{combinedRH-zero}) and (\ref{combinedRH-xy-1}), we find
\begin{equation}\label{combinedRH-xy-zero}
\Psi(0, 0, 0, x,y) = F(0,0,0, u_2+(c_2-x)\cos(y+\theta_w))=0
\qquad\text{on}\,\, \Shock\cap\{0<x<\eps_0\}.
\end{equation}

By its explicit definition
(\ref{combinedRH-1-1})--(\ref{combinedRH-1-2}),
(\ref{combinedRH-2}), and (\ref{combinedRH-xy-1}), the function
$\Psi(p_1, p_2, p_3, x, y)$ is $C^\infty$ on the set $\{|(p_1, p_2,
p_3, x)|<\delta\}$, where $\delta>0$ depends only on $u_2, v_2,
\rho_2, \xi_0, \eta_0$, i.e., on the data. Using
(\ref{psiDerivEstNearSonic}) and choosing $\eps>0$ small, we obtain
$$
|x|+|\psi(x,y)|+|D\psi(x,y)|\le \delta \qquad \text{for all
$(x,y)\in \overline{\Om_{\eps}}$}.
$$
Thus, from (\ref{combinedRH-xy})--(\ref{combinedRH-xy-zero}),
it follows that $\psi$ satisfies (\ref{combinedCondOnShock}) on
$\Shock\cap\{0<x<\eps \}$, where
\begin{equation}\label{coefFormulas}
\hat b_k(x,y)=\int_0^1 \Psi_{p_k}\Big(t\psi_x(x,y),t\psi_y(x,y),
t\psi(x,y), x,y\Big)\,dt \qquad\text{for }\;k=1,2,3.
\end{equation}
Thus, we have
$$
\hat b_k\in C(\overline{\Shock\cap\{0\le x\le\eps \}}),\quad |\hat
b_k|\le \frac{1}{\lambda} \qquad\text {on }\; \Shock\cap\{0<x<\eps
\}, \,\,\,\text {for }\;k=1,2,3,
$$
for some $\lambda>0$.

It remains to show that $\hat b_1\ge \lambda$ for some $\lambda >0$.
For that, since $\hat b_1$ is defined by (\ref{coefFormulas}), we
first show that $\Psi_{p_1}(0, 0, 0, 0, y_1)>0$, where $( x_1,
y_1)=(0, \hat f(0))$  are the coordinates of     $\PtUpL$.

In the calculation, we will use that, since  $(0, y_1)$ are the
$(x,y)$-coordinates of  $\PtUpL=(\xi_1, \eta_1)$, then, by
(\ref{coordPolar})--(\ref{coordNearSonic}),
$$
\xi_1=u_2+c_2\cos(y_1+\theta_w), \qquad
\eta_1=v_2+c_2\sin(y_1+\theta_w),
$$
which implies
$$
(\xi_1-u_2)^2+(\eta_1-v_2)^2=c_2^2.
$$
Also, $c_2^2=\rho_2^{\gamma-1}$.
Then, by explicit calculation, we obtain
\begin{equation}\label{4.27a}
\begin{split}
\Psi_{p_1}(0, 0, 0, 0, y_1)=&
\frac{\rho_1}{c_2}\big((u_1-\xi_1)(\xi_1-u_2)-\eta_1(\eta_1-v_2)\big)\\
&-\frac{\rho_2}{c_2}\big((u_2-\xi_1)(\xi_1-u_2)+(v_2-\eta_1)(\eta_1-v_2)\big).
\end{split}
\end{equation}
Now, working in the $(\xi,\eta)$-coordinates on the right-hand side
and noting that $\grad\vphi_1(\xi_1,\eta_1)=(u_1-\xi_1,-\eta_1)$ and
$\grad\vphi_2(\xi_1,\eta_1)=(u_2-\xi_1,v_2-\eta_1)$, we rewrite
\eqref{4.27a} as
\begin{equation*}
\begin{split}
&\Psi_{p_1}(0, 0, 0, 0, y_1)=-\frac 1{c_2}
\Big(\rho_1\grad\vphi_1(\xi_1,\eta_1)- \rho_2\grad\vphi_2(\xi_1,\eta_1)
\Big)\cdot \grad\vphi_2(\xi_1,\eta_1),
\end{split}
\end{equation*}
where $\grad=(\partial_\xi, \partial_\eta)$. Since the point
$\PtUpL$ lies on the shock $S_1=\{\vphi_1=\vphi_2\}$ separating
state (2) from state (1), then, denoting by $\tau_0$ the unit vector
along the line $S_1$, we have
$$
\grad\vphi_1(\xi_1,\eta_1)\cdot \tau_0=\grad\vphi_2(\xi_1,\eta_1)\cdot \tau_0.
$$
Now, using the Rankine-Hugoniot condition
(\ref{FBConditionSelfSim-0}) at the point $\PtUpL$ for $\vphi_1$ and
$\vphi_2$, we obtain
$$
\rho_1\grad\vphi_1(\xi_1,\eta_1)- \rho_2\grad\vphi_2(\xi_1,\eta_1)=
(\rho_1-\rho_2)\big(\grad\vphi_2(\xi_1,\eta_1)\cdot \tau_0\big)
\tau_0,
$$
and thus
\begin{equation*}
\begin{split}
&\Psi_{p_1}(0, 0, 0, 0, y_1)= \frac 1{c_2}
(\rho_2-\rho_1)\big(\grad\vphi_2(\xi_1,\eta_1)\cdot \tau_0\big)^2,
\end{split}
\end{equation*}
where $\rho_2>\rho_1$ by the assumption of our theorem.

Thus it remains to prove that $\grad\vphi_2(\xi_1,\eta_1)\cdot
\tau_0\ne 0$. Note that $|\grad \vphi_2(\xi_1,\eta_1)|=c_2=
\rho_2^{(\gamma-1)/2}$, since $(\xi_1,\eta_1)$ is on the sonic
circle. Thus, on the contrary, if $\grad\vphi_2(\xi_1,\eta_1)\cdot
\tau_0= 0$, then, using also $\grad\vphi_1(\xi_1,\eta_1)\cdot
\tau_0=\grad\vphi_2(\xi_1,\eta_1)\cdot \tau_0$, we can write the
Rankine-Hugoniot condition (\ref{FBConditionSelfSim-0}) at $(\xi_1,
\eta_1)$ in the form:
\begin{equation}\label{RHsonic}
\rho_1|\grad\vphi_1(\xi_1,\eta_1)|=\rho_2\,
\rho_2^{(\gamma-1)/2}=\rho_2^{(\gamma+1)/2}.
\end{equation}
Since both $\vphi_1$ and $\vphi_2$ satisfy (\ref{1.1.5})
and since
$\vphi_1(\xi_1,\eta_1)=\vphi_2(\xi_1,\eta_1)$ and $|\grad
\vphi_2(\xi_1,\eta_1)|=c_2$, we have
$$
\rho_1^{\gamma-1}+\frac{\gamma-1}{2}|\grad\vphi_1(\xi_1,\eta_1)|^2=
\rho_2^{\gamma-1}+\frac{\gamma-1}{2}\rho_2^{\gamma-1}.
$$
Combining this with (\ref{RHsonic}), we obtain
\begin{equation}\label{RH-Bern-sonic}
\frac 2{\gamma+1} \left(\frac {\rho_1}{\rho_2} \right)^{\gamma-1}
+\frac {\gamma-1}{\gamma+1}\left(\frac {\rho_2}{\rho_1} \right)^2 =1.
\end{equation}
Consider the function
$$
g(s)=\frac 2{\gamma+1} s^{\gamma-1} +\frac
{\gamma-1}{\gamma+1}s^{-2} \qquad\text{on}\,\, (0,\infty).
$$
Since $\gamma>1$, we have
$$
g'(s)<0 \quad \text{on}\, (0, 1); \qquad\,\, g'(s)>0
\quad\text{on}\, (1, \infty); \qquad\,\, g(1)=1.
$$
Thus, $g(s)=1$ only for $s=1$. Therefore, (\ref{RH-Bern-sonic})
implies $\rho_1=\rho_2$, which contradicts the assumption
$\rho_1<\rho_2$ of our theorem. This implies that
$\grad\vphi_2(\xi_1,\eta_1)\cdot \tau_0\ne 0$, thus $\Psi_{p_1}(0,
0, 0, 0, y_1)>0$.

Choose $\lambda:=\frac 12 \Psi_{p_1}(0, 0, 0, 0, y_1)$. Then
$\lambda>0$. Since the function $\Psi(p_1, p_2, p_3,  x, y)$ is
$C^\infty$ on the set $\{|(p_1, p_2, p_3,  x)|<\delta\}$ and since
$\psi\in C^{1,1}(\overline{\Om_{\eps_0}})$ with
$\psi(0,0)=\psi_x(0,0)=\psi_y(0,0)=0$ by (\ref{psiZeroOnSonic}), we
find that, for small $\eps>0$,
$$
\Psi_{p_1}\Big(t\psi_x(x,y),t\psi_y(x,y), t\psi(x,y),
x,y\Big)\ge\lambda \quad \text{for all}\,\,\,
(x,y)\in\Shock\cap\{0<x<\eps \},\;t\in [0,1].
$$
Thus, from (\ref{coefFormulas}), we find  $\hat b_1\ge \lambda$.
Lemma \ref{bdryCondShock-xy} is proved.
\end{proof}

\medskip
Step 7. {\bf A sequence close to $\Gamma_{shock}$}. Now we construct
the sequence close to $\Shock$. Recall that we have assumed that
assertion (\ref{noLimAtP1}) is false, i.e., $D^2\psi$ has a limit
 at $\PtUpL$ from $\Om$. Then
 (\ref{seqCloseSonic}) implies
\begin{equation}\label{limitPsiXYisZero}
\lim_{(x,y)\to (0,\;\hat{f}(0))\atop
(x,y)\in\Om}
\psi_{xy}(x,y)
=\lim_{(x,y)\to (0,\;\hat{f}(0))\atop
(x,y)\in\Om}
\psi_{yy}(x,y)
=0,
\end{equation}
where $(0,\;\hat{f}(0))$ are the coordinates of $\PtUpL$ in the
$(x,y)$-plane. Note that, from (\ref{psiZeroOnSonic}),
$$
\psi_y(x, \hat{f}(x))
=\int_0^x \psi_{xy}(s,\hat{f}(0))ds
+\int_{\hat{f}(0)}^{\hat{f}(x)}
\psi_{yy}(x,t)dt,
$$
and, from (\ref{OmegaXY}), all points in the paths of integration
are within $\Om$. Furthermore, by (\ref{fhat}),
$0<\hat{f}(x)-\hat{f}(0)<Cx$ with $C$ independent of $x\in (0,
\eps_0)$. Now, (\ref{limitPsiXYisZero}) implies
\begin{equation}\label{limitPsiYisZeroX}
\lim_{x\to 0+}
\frac{\psi_{y}(x,\hat{f}(x))}x
=0.
\end{equation}

Also, by Lemma \ref{bdryCondShock-xy},
$$
|\psi_{x}(x,\hat{f}(x))|= |\frac{\hat{b}_2\psi_y+\hat{b}_3\psi}{\hat{b}_1}|\le
C(|\psi_y|+|\psi|)
\qquad\text{on}\;\;(0,\eps),
$$
where $\eps>0$ is from Lemma \ref{bdryCondShock-xy}. Then, using
(\ref{limitPsiYisZeroX}) and  $|\psi(x,y)|\le Cx^2$ by
(\ref{psiDerivEstNearSonic}), we have
\begin{equation}\label{limitPsiXisZeroX}
\lim_{x\to 0+}
\frac{\psi_{x}(x,\hat{f}(x))}x
=0.
\end{equation}

\medskip
Let
$$
\mcl{F}(x):=\psi_x(x,\;\hat{f}(x)-\frac{\omega}{10}x)
$$
for some constant $\omega>0$. Then $\mcl F(x)$ is well-defined and
differentiable for $0<x<\eps_0$ so that
\begin{equation}
\begin{split}
  \mcl{F}(x)&=\psi_{x}(x,\;\hat{f}(x)-\frac{\omega}{10}x)\\
      &=\psi_{x}(x,\hat{f}(x))+\int_0^1\frac{d}{dt}\psi_{x}
       (x,\;\hat{f}(x)-\frac{t\omega}{10}x)\, dt\\
      &=\psi_{x}(x,\hat{f}(x))-\frac{\omega}{10} x\int_0^1
      \psi_{xy}(x,\hat{f}(x)- \frac{t\omega}{10}x)\, dt.
\end{split}
\end{equation}
Now (\ref{limitPsiXYisZero}) and (\ref{limitPsiXisZeroX}) imply
\begin{equation}\label{limit}
\lim_{x\to 0+}
\frac{\mcl{F}(x)}x
=0.
\end{equation}

 By (\ref{regPsi}) and since $\hat f\in C^{1,1}([0,\eps_0])$, we have
 \begin{equation}\label{reg-F}
 \mcl{F}\in C([0, \eps])\cap C^1((0,\eps)).
 \end{equation}
Then (\ref{limit}) and the mean-value theorem imply that there
exists a sequence $\{x_k^{(2)}\}$ with $x_k^{(2)}\in (0, \eps)$ and
 \begin{equation}\label{seqNearShock-x}
 \lim_{k\to\infty} x_k^{(2)}=0
 \quad\text{and}\quad
 \lim_{k\to\infty}\mcl{F}'(x_k^{(2)})=0.
 \end{equation}
%
%
%
%

By definition of $\mcl{F}(x)$,
\begin{equation}\label{3.40}
\psi_{xx}\big(x,\;g(x)\big)= \mcl{F}'(x)-g'(x) \psi_{xy}(x,\; g(x))
\end{equation}
where $g(x):=\hat{f}(x)-\frac{\omega}{10}x$.

On the other hand, $|\hat{f}'(x)|$ is bounded. Then, using
(\ref{limitPsiXYisZero}) and (\ref{seqNearShock-x})--(\ref{3.40})
yields
$$
\lim_{k\rightarrow \infty}\psi_{xx}(x_k^{(2)},g(x_k^{(2)}))=
\lim_{k\rightarrow \infty}\mcl{F}'(x_k^{(2)})=0.
$$
Note that $\lim_{x\to 0+} g(x)=\hat f(0)$. Thus, denoting
$y_k^{(2)}=g(x_k^{(2)})$, we conclude
$$
(x_k^{(2)}, y_k^{(2)})\in\Om,\quad
\lim_{k\rightarrow \infty}(x_k^{(2)}, y_k^{(2)})=(0,\hat f(0)),\quad
\lim_{k\rightarrow \infty}\psi_{xx}(x_k^{(2)},y_k^{(2)})=0.
$$
Combining this with (\ref{seqCloseSonic}), we conclude that
$\psi_{xx}$ does not have a limit  at $\PtUpL$ from $\Om$, which
implies assertion (\ref{noLimAtP1}). This completes the proof of
Theorem \ref{mainthm1}.
\end{proof}

\begin{remark}
For the isothermal case, $\gam=1$, there exists a global regular
reflection solution in the sense of Definition
{\rm{\ref{RegReflSolDef}}} when $\theta_w\in (0,\frac{\pi}{2})$ is
close to $\frac{\pi}{2}$. Moreover, the solution has the same
properties stated in Theorem {\rm\ref{mainthm1}} with $\gamma=1$.
This can be verified by the limiting properties of the solutions for
the isentropic case when $\gamma\to 1+$. This is because, when
$\gamma\to 1+$,
$$
i(\rho)\to ln \rho, \qquad p(\rho)\to \rho, \qquad c^2(\rho)\to 1
\qquad\quad  \mbox{in \eqref{gamma-law}},
$$
$$
\rho(|D\varphi|^2, \varphi)\to \rho_0
e^{-(\varphi+\frac{1}{2}|D\varphi|^2)} \qquad\quad \mbox{in
\eqref{1.1.6}},
$$
and
$$
c_*(\varphi, \rho_0,\gamma)\to 1 \qquad\quad \mbox{in
\eqref{1.1.8a}},
$$
in which case the arguments for establishing Theorem {\rm
\ref{mainthm1}} is even simpler.
\end{remark}

\bigskip
{\bf Acknowledgments.} The authors thank Luis Caffarelli for helpful
suggestions and comments. This paper was completed when the authors
attended the ``Workshop on Nonlinear PDEs of Mixed Type Arising in
Mechanics and Geometry'', which was held at the American Institute
of Mathematics, Palo Alto, California, March 17--21, 2008. Gui-Qiang
Chen's research was supported in part by the National Science
Foundation under Grants DMS-0505473, DMS-0244473, and an Alexander
von Humboldt Foundation Fellowship. Mikhail Feldman's research was
supported in part by the National Science Foundation under Grants
DMS-0500722 and DMS-0354729.


\bibliographystyle{amsplain}

\end{document}